\documentclass[11pt]{article}
\usepackage[a4paper]{geometry}
\usepackage{graphicx} 
\usepackage{amsmath}
\usepackage{mathrsfs}
\usepackage{amsthm}
\usepackage{titleps}
\usepackage{mathtools}
\usepackage{amssymb}
\usepackage{booktabs}
\usepackage{multicol}
\usepackage{enumitem}
\usepackage{diagbox}
\usepackage{array}
\usepackage[hang,small,bf]{caption}
\usepackage[subrefformat=parens]{subcaption}
\usepackage{tikz}             
\usetikzlibrary{arrows.meta}  
\usetikzlibrary{positioning}  
\usetikzlibrary{calc}         
\usetikzlibrary{intersections, hobby}
\usepackage{xcolor}

\usepackage{pgfplots} 
\pgfplotsset{compat=1.18} 
\usepgfplotslibrary{fillbetween}
\usetikzlibrary{intersections,calc}
\usetikzlibrary{patterns}

\usepackage{tikz-cd}

\usepackage{booktabs,tabularx,array,multirow,diagbox,multicol}
\usepackage{adjustbox}

\theoremstyle{plain}
\newtheorem{thm}{Theorem}[section]
\newtheorem{prop}[thm]{Proposition}

\newtheorem{ques}[thm]{Question}
\newtheorem{lem}[thm]{Lemma}
\newtheorem{fact}[thm]{Fact}

\newtheorem*{bp*}{Basic Problem}
\theoremstyle{definition}
\newtheorem{dfn}[thm]{Definition}
\theoremstyle{remark}
\newtheorem{rmk}[thm]{Remark}

\newtheorem*{claim*}{Claim}
\newtheorem{claim}[thm]{Claim}
\newtheorem{conv}[thm]{Convention}

\renewcommand{\l}{\mathfrak{l}}
\newcommand{\Lin}{\mathscr{L}}
\newcommand{\U}{\mathscr{U}}
\newcommand{\s}{\mathfrak{s}}
\newcommand{\n}{\mathfrak{n}}
\newcommand{\m}{\mathfrak{m}}
\newcommand{\R}{\mathbb{R}}
\newcommand{\N}{\mathbb{N}}

\newcommand{\SR}{SL_2(\R)\ltimes \R^2}
\newcommand{\sr}{\mathfrak{sl}_2(\R)\ltimes \R^2}
\renewcommand{\sl}{\mathfrak{sl}_2(\R)}
\newcommand{\SL}{SL_2(\R)}
\newcommand{\GR}{GL_2(\R)\ltimes \R^2}
\newcommand{\GL}{GL_2(\R)}
\newcommand{\pf}{\pitchfork}
\newcommand{\Hfir}{Z}
\newcommand{\Hsec}[1]{A({#1})}
\newcommand{\Hthi}{U}
\newcommand{\Hfou}[1]{B({#1})}
\newcommand{\Hfif}{B^\prime}
\newcommand{\Hsix}{D}
\newcommand{\hfir}{\mathfrak{z}}
\newcommand{\hsec}[1]{\mathfrak{a}(#1)}
\newcommand{\hthi}{\mathfrak{u}}
\newcommand{\hfou}[1]{\mathfrak{b}({#1})}
\newcommand{\hfif}{\mathfrak{b}^\prime}
\newcommand{\hsix}{\mathfrak{d}}
\renewcommand{\a}{\mathfrak{a}}
\newcommand{\pmat}[1]{\begin{pmatrix}#1\end{pmatrix}}
\newcommand{\Vmat}[1]{\begin{Vmatrix}#1\end{Vmatrix}}
\newcommand{\spnR}[1]{\langle #1\rangle}

\renewcommand{\arraystretch}{1.5}

\usepackage{titleps}

\makeatletter
\newcommand{\MSC}[1]{\gdef\@MSC{#1}}
\newcommand{\keywords}[1]{\gdef\@keywords{#1}}

\title{Classification of connected proper pairs in the affine transformation group}

\author{Shunsuke Miyauchi}
\date{\today}

\begin{document}

\maketitle
\begin{abstract}
    Let $(L, H)$ be closed subgroups of a locally compact group $G$. The pair $(L, H)$ is said to be \emph{proper} if the action of $L$ on the homogeneous space $G/H$ is proper.

    We give a complete list of connected closed proper pairs in the affine transformation group of $\R^2$. This result extends Kobayashi's classification of connected closed subgroups of the affine transformation group of $\R^2$ acting properly on $\R^2$.
\end{abstract}

\begin{flushleft}
\end{flushleft}

\tableofcontents

\section{Introduction}\label{sec:intro}
Consider a pair $(L,H)$ of closed subgroups of a Lie group $G$. In the theory of transformation groups, it is crucial to find effective criteria for properness (see the survey \cite{Kobsurv} by Kobayashi, who initiated this line of study). We say that the pair $(L,H)$ is proper if $L$ acts properly on the homogeneous space $G/H$; equivalently, $H$ acts properly on $G/L$. The pair $(L,H)$ is always proper if either $L$ or $H$ is compact. However, when both $L$ and $H$ are non-compact, determining whether the pair $(L,H)$ is proper becomes difficult.

Following the insights presented in Kobayashi~\cite{Kobpre,Kobsurv}, we briefly review previous works.

The first computable criterion was established by Kobayashi~\cite{Kob89} when $G,L,H$ are reductive. This result was later generalized by Benoist~\cite{Ben96} and Kobayashi~\cite{Kob96} to the case where $G$ is reductive, but $L,H$ are not necessarily reductive.

Lipsman~\cite{Lip95} considered whether an analog of the properness criterion for reductive groups holds for nilpotent groups. In the setting where $G$ is nilpotent and $L$ and $H$ are connected, he conjectured that the (CI) condition (see Definition~\ref{def:CI}) is equivalent to properness. The (CI) condition, which is defined by Kobayashi~\cite{Kob90}, is weaker than properness in general and easier to check. This conjecture was solved affirmatively by Nasrin~\cite{Nas01}, Yoshino~\cite{Yos07b}, Baklouti-Khlif~\cite{BK05} for $n$-step nilpotent groups with $n\leq 3$, but Yoshino~\cite{Yos05} gave a counterexample in the $n$-step case for $n=4$.

On the other hand, there is no verifiable criterion when
$G$ is neither reductive nor nilpotent. In this setting, we consider the following question (Subsection~\ref{subsection:mainresult}):
\begin{ques}
Classify proper pairs $(L,H)$ of connected closed subgroups of the affine transformation group of $\R^2$.
\end{ques}

We briefly review the affine transformation group.  The affine transformation group of $\R^n$, namely $GL_n(\R)\ltimes \R^n$, is realized as the subgroup of $GL_{n+1}(\R)$ given by
\[
GL_n(\R)\ltimes \R^n \coloneqq \left\{ \begin{pmatrix} g & v \\ 0 &1\end{pmatrix} \colon g \in GL_n(\R),\, v \in \R^n \right\}.
\]
We denote such an element $\begin{pmatrix} g & v \\0 &1 \end{pmatrix}$  by $( g , v)$, and refer to $g$ as the linear part and $v$ as the translational part.
For the element $(g , v) \in GL_n(\R) \ltimes \R^n$, we define the natural projections $\Lin \colon GL_n(\R) \ltimes\R^n \to GL_n(\R), \U \colon GL_n(\R)\ltimes\R^n \to \R^n$ by
\[ \Lin(g , v) \coloneqq g, \quad \U(g , v) \coloneqq v.\]

\subsection{The case $G=\SR$}\label{subsection:properSR}
When $G=\SR$, there are only finitely many equivalence classes of connected subgroups with respect to the relation $\sim$  (Definition~\ref{def:proper}). Hence the classification is relatively simple. In contrast, when $G=\GR$, the set of such equivalence classes has the cardinality of the continuum.
\begin{thm}\label{thm:properSR}
\begin{enumerate}
\item Up to the equivalence relation $\sim$, the non-trivial connected subgroups are as follows:
\[ \R^2,\SL,S,L,M,N,\]
where $S,L,M,N$ are defined below.
\item The properness of pairs among $\SL,$ $\R^2,$  $S,$ $L,$ $M,$ and $N$ is summarized in Table~\ref{tab:properSR}. In Table~\ref{tab:properSR}, we mark an entry with $\pf$ if the corresponding pair is proper; otherwise the entry is left blank.
\end{enumerate}
\end{thm}
Here $S,L,M,N$ denote the following subgroups of $\SR$:
\begin{align*}
 S&\coloneqq \{ \begin{pmatrix} e^a & 0 & e \\ 0 & e^{-a} & 0  \end{pmatrix}\colon a,e \in \R\}, 
 L \coloneqq  \{ \begin{pmatrix} 1 & 0 & 0 \\ b & 1 & f  \end{pmatrix}\colon b,f \in \R\},\\
 M &\coloneqq  \{ \begin{pmatrix} 1 & 0 & b \\ b & 1 & e  \end{pmatrix}\colon b,e \in \R\} \supset N \coloneqq \{ \begin{pmatrix} 1 & 0 & b \\ b & 1 & \frac12 b^2  \end{pmatrix}\colon b \in \R\}.
 \end{align*}
 
 Recall that $\pf$ is the binary relation introduced by Kobayashi~\cite{Kob96} generalizing the concept of proper actions (Definition~\ref{def:proper}). 

\begin{table}[hbtp]
\caption{Properness of pairs of connected subgroups of $\SR$}\label{tab:properSR}
\centering
\renewcommand{\arraystretch}{1.5}
\begin{tabular}{c@{\quad}cccccc}
\toprule
 \diagbox{$L$}{$H$}& $SL_2(\R)$ & $\R^2$ & $S$ & $L$ & $M$ & $N$ \\
\midrule
$SL_2(\R)$ &  & $\pf$ &  &  & $\pf$ & $\pf$ \\
$\R^2$ & $\pf$ &  &  &  &  & $\pf$ \\
$S$ &  &  &  &  &  & $\pf$ \\
$L$ &  &  &  &  &  &  \\
$M$ & $\pf$ &  &  &  &  &  \\
$N$ & $\pf$ & $\pf$ & $\pf$ &  &  &  \\
\bottomrule
\end{tabular}
\end{table}

\subsection{The case $G=\GL$}\label{subsection:properGL}
In this case, proper pairs can be classified by a direct application of the properness criterion  \cite{Kob89,Kob96,Ben96} for reductive groups.

\begin{thm}\label{thm:properGL}
\begin{enumerate}
\item Up to the equivalence relation $\sim$, the non-trivial connected subgroups of $\GL$ are as follows:
\[ Z,U,A(\alpha),B(\alpha),\]
where $\alpha \in \R$ and $Z,U,A(\alpha),B(\alpha)$ are defined below.
\item The properness of the pairs among the connected subgroups $Z,U,A(\alpha),B(\alpha)$ is described in the following Table~\ref{tab:properGL}. Here, each entry in the table that specifies a condition on the parameters $\alpha ,\beta$ provides a necessary and sufficient condition for the corresponding pair to be proper.
    \end{enumerate}
\end{thm}
\begin{table}[htbp]
\centering
\caption{Properness of the pairs in $\GL$}
\label{tab:properGL}
\begin{tabular}{c@{\quad}cccc}
    \toprule
     \diagbox{$L$}{$H$} &  $Z$ & $A(\beta)$ & $U$ & $B(\beta)$  \\
    \midrule
    $Z$ & & $\pf$ & $\pf$ & $\pf$   \\
    $A(\alpha)$ &$\pf$ & $|\alpha| \neq |\beta|$ & $\pf$ & $|\alpha|>|\beta|$  \\
    $U $  &$\pf$ & $\pf$ &  & $\pf$   \\
    $B(\alpha)$  & $\pf$ & $|\alpha|<|\beta|$ & $\pf$ &   \\
    \bottomrule
    \end{tabular}
    \end{table}

Here $Z,U,A(\alpha),B(\alpha)$ denote the following subgroups of $\GL$, where $\alpha$ is a real constant.
\[ Z\coloneqq \left\{ \begin{pmatrix} e^t & 0 \\ 0 & e^t\end{pmatrix}\colon t\in \R\right\}, 
            U\coloneqq \left\{ \begin{pmatrix} e^t & 0 \\ te^t & e^t\end{pmatrix}\colon t \in \R \right\}, \]
            \[ A(\alpha)\coloneqq \left\{ \begin{pmatrix} e^{(\alpha +1)t} & 0 \\ 0 & e^{(\alpha -1)t}\end{pmatrix} \colon t \in \R\right\}\subset
            B(\alpha) \coloneqq \left\{ \begin{pmatrix} e^{(\alpha +1)t} & 0 \\ s & e^{(\alpha -1)t}\end{pmatrix}\colon s,t \in \R\right\}. \]

\subsection{Main result (the case $G=\GR$)}\label{subsection:mainresult}
As a test case beyond the setting where $G$ is reductive or nilpotent, Kobayashi~\cite[Prop.\ A.2.1]{Kob90} classified all pairs $(L,H)$ of connected subgroups that are proper in $G=\GR$ when $H=\GL$.
We extend Kobayashi's setting to arbitrary connected subgroups $H \subset \GR$. Table~\ref{tab:howtodetermineproperness} summarizes a procedure for determining properness of a pair $(L,H)$ from the types of the linear parts $\Lin(L)$ and $\Lin(H)$.

{
\renewcommand\arraystretch{1.5}
\setlength{\extrarowheight}{1.2pt}

\begin{table}[htbp]
  \centering
  \captionsetup{font=small}
  \caption{Procedure for determining properness from the linear parts}
  \label{tab:howtodetermineproperness}

  \begin{adjustbox}{max width=\linewidth}
  \begin{tabularx}{\linewidth}{
     >{\raggedright\arraybackslash}p{18mm}  
        >{\centering\arraybackslash}p{10mm}     
        >{\centering\arraybackslash}p{10mm}    
        >{\centering\arraybackslash}p{14mm} 
        *{4}{>{\centering\arraybackslash}X}     
  }
    \toprule
     \diagbox[width=23mm, innerleftsep=0pt, innerrightsep=0pt]{$\Lin(L)$}{$\Lin(H)$}
      & $\GL$
      & $D~\text{or}~B^\prime$
      & \parbox{17mm}{\vspace*{0.3ex}\centering contained \\  in $\SL$}
      & $Z$
      & $A(\beta)$
      & $U$
      & $B(\beta)$ \\
    \midrule\midrule
     $\GL$ & \multicolumn{7}{@{}c@{}}{\cite[Prop.\ A.2.1]{Kob90}} \\
     $D~\text{or}~B^\prime$ & - & \multicolumn{6}{@{}c@{}}{Theorem~\ref{prop:LinL=DB'}} \\
     \parbox{17mm}{\vspace*{0.3ex}\centering contained  \\in $\SL$}& - & - & Theorem~\ref{thm:properSR} & \multicolumn{4}{@{}c@{}}{Theorems~\ref{thm:properSR} and \ref{thm:slded}} \\
     $Z$ &- &- &- & $\not\pf$ & $\circledcirc$& $\circledcirc$& $\circledcirc$\\
     $A(\alpha)$ & - & -& - & - & {\bf (A)} &  $\circledcirc$ & {\bf (C)} \\ 
     $U$ & - & - & - & - & - & $\not\pf$ & $\circledcirc$ \\
     $B(\alpha)$ & - & - & - & - & - & - & {\bf (B)} \\
    \bottomrule
  \end{tabularx}
  \end{adjustbox}
\end{table}
}

We explain below the symbols in Table~\ref{tab:howtodetermineproperness} and state the theorems cited there. 

The symbol $\circledcirc$  means that we apply Theorems \ref{thm:properGL} and \ref{thm:RCI}.

The subgroups $Z,U,A(\alpha),B(\alpha)$ of $\GL (\alpha \in \R)$ are defined in Subsection~\ref{subsection:properGL}, and their dimensions are $1,1,1,2$ respectively. The connected subgroups  $D,B^\prime$ of $\GL$  are defined as follows:
\[ D\coloneqq \left\{ \begin{pmatrix} e^t & 0 \\ 0 & e^s \end{pmatrix}\colon s,t \in \R \right\}, B^\prime\coloneqq \left\{ \begin{pmatrix} e^t & 0 \\ s & e^t \end{pmatrix}\colon s,t \in \R \right\}.\]

\begin{thm}\label{prop:LinL=DB'}
Let $L$ and $H$ be connected closed subgroups of $\GR$. If $L \pf H$ in $G$, then $(L,H)$ is of the following form:
\[ (L(\Hsix,1),N)\]
where $L$ is $L(\Hfif,1)$ or $L(\Hsix,1)$.
\end{thm}
The subgroup $N$ is the one-dimensional connected subgroup of $\SR$ defined in Subsection~\ref{subsection:properSR}, and $L(\Hfif,1)$ and $L(\Hsix,1)$ are defined below.
\[ L(\Hfif,1)=\left\{ \begin{pmatrix} e^{t} & 0 & 0 \\ s & e^{t} & u \end{pmatrix}\colon s,t,u \in \R\right\}, L(\Hsix,1)=   \left\{ \begin{pmatrix} e^{t+s} & 0 & 0 \\ 0 & e^{t-s} & u \end{pmatrix}\colon s,t,u \in \R\right\}.\]
\begin{thm}\label{thm:slded}
    Suppose that $L \subset SL_n(\R) \ltimes\R^n$ is a  subset and $H \not\subset SL_n(\R)\ltimes \R^n$ is a connected subgroup of $GL_n(\R)\ltimes \R^n$. Then, we have the following equivalence:
\[ L \pitchfork H \ \text{in  }  GL_n(\R) \ltimes \R^n \iff L\pitchfork (H\cap SL_n(\R) \ltimes \R^n) \text{ in } SL_n(\R) \ltimes \R^n.\]
\end{thm}

\begin{thm}\label{thm:RCI}
    Let $L$ be a closed subgroup in $GL_n(\R)\ltimes \R^n$ whose linear part $\Lin(L)$ is closed, and $H$ be a subset in $GL_n(\R)\ltimes \R^n$. If the pair $(L,\R^n)$ is (CI) and $\Lin(L) \pf \Lin(H)$ in $GL_n(\R)$, then $L \pf H$ in $GL_n(\R)\ltimes \R^n$.
\end{thm}

The conditions {\bf(A), (B)} and {\bf(C)} in Table~\ref{tab:howtodetermineproperness} are described in the following theorem. This theorem is the main technical result of this paper.
\begin{thm}\label{thm:conditionABC} Let $L$ and $H$ be connected closed subgroups of $\GR$.
    \begin{description}
        \item[{\bf(A)}] Assume that $\Lin(L)=\Hsec\alpha$ and $\Lin(H)=\Hsec{\beta}$. Then, the (CI) condition is equivalent to properness of the pair.
        \item[{\bf(B)}] Assume that $\Lin(L)=\Hfou\alpha$ and $\Lin(H)=\Hfou\beta$. If the pair $(L,H)$ is proper, then at least one of $L$ and $H$ is $L(\Hfou{-3},3)$ or $L(\Hfou{-3},6)$. Furthermore, if $H$ is these subgroups, then the following equivalence holds:
        \[ L\pf H \iff (L,H) \text{ is (CI) and } |\alpha|<3.\]
    \item[{\bf(C)}] Assume that $\Lin(L)=\Hsec\alpha$ and $\Lin(H)=\Hfou\beta$. Then, the properness of the pairs $(L,H)$ are described in Table~\ref{tab:conditionC}. Here, each entry in the table that specifies a condition on the parameters $\alpha ,\beta$ provides a necessary and sufficient condition for the corresponding pair to be proper.
    \end{description}
\end{thm}
Here, the connected subgroups $L(\Hsec{\alpha},i), L(\Hfou{\beta},j)$ are defined in Section~\ref{sec:clGR}.

This theorem is proved by using the classification of connected subgroups of $\GR$ due to \cite[Thm.\ 11]{CKZ24} and by checking, case by case, the relation $\pf$ (Definition~\ref{def:proper}). Since the ambient group $\GR$ is low-dimensional, this verification requires a relatively small amount of computation. In general, it is difficult to check the binary relation $\pf$ directly.
 \begin{table}[hbtp]
\caption{The case $\Lin(L)=\Hsec{\alpha}$ and $\Lin(H)=\Hfou\beta$}
\label{tab:conditionC}
\centering
\setlength{\tabcolsep}{5pt}
\begin{adjustbox}{max width=\linewidth} 
\begin{tabular}{@{}c*{7}{c}@{}}
\toprule
\diagbox{$L$}{$H$} 
& $L(B(\beta),1)$ 
& $L(B(1),2)$ 
& $L(B(-3),3)$ 
& $L(B(\beta),4)$ 
& $L(B(-1),5)$ 
& $L(B(-3),6)$ 
& $L(B(\beta),7)$ \\
\midrule
$L(A(\alpha),1)$ 
& $|\alpha|>|\beta|$ 
& $|\alpha|\geq1$ 
& $|\alpha|\ne 3$ 
& $|\alpha|>|\beta|$ 
& $\alpha\ne 0$ 
& $|\alpha|\ne 3$ 
& $|\alpha|>|\beta|$ \\
$L(A(1),2)$ 
& $\pf$ 
&  
& $\pf$ 
& $-1 \le \beta < 1$ 
&  
& $\pf$ 
& $1 > |\beta|$ \\
$L(A(\alpha),3)$ 
& $|\alpha|>|\beta|$ 
& $|\alpha|>1$ 
& $|\alpha|\ne 3$ 
&  
&  
&  
&  \\
$L(A(1),4)$ 
& $\pf$ 
&  
&  $\pf$
&  
&  
&  
&  \\
$L(A(\alpha),5)$ 
& $|\alpha|>|\beta|$ 
& $|\alpha|>1$ 
& $|\alpha|>3$ 
&  
&  
&  
&  \\
\bottomrule
\end{tabular}
\end{adjustbox}
\end{table}

\subsection{Organization of the paper}
In Section~2, we review basic properties of $\pf$ and $\sim$, which are introduced in \cite{Kob96}. The binary relation $\pf$ is a generalization of proper actions, and $\sim$ is the coarsest equivalence relation which preserves the relation $\pf$.  

In Section~3, we introduce the notations $\lesssim_+, \asymp_+, \lesssim_\times$, and $\asymp_\times$.
These notations are used to analyze asymptotic behavior and to verify the relations $\pf$ and $\sim$, mainly in the proofs of Theorems~\ref{thm:properSR} and~\ref{thm:conditionABC}.

In Sections 4 and 5, we prove Theorems~\ref{thm:properSR} and \ref{thm:properGL} respectively.

In Section~6, we review the classification of the connected subgroups in $\GR$, which is proved in \cite[Thm.\ 11]{CKZ24}.

In Section~7, we prove Theorems~\ref{prop:LinL=DB'} - \ref{thm:RCI}. 

In the last section, we prove Theorem~\ref{thm:conditionABC}.

In Appendix~A, we provide a complete list of pairs of connected  closed subgroups $(L,H)$ such that $L$ acts properly on  $G/H$, that is, $L\pf H$ in $G$, up to $\sim$.

\subsection*{Acknowledgements}
The author expresses sincere gratitude to Professor Toshiyuki Kobayashi for his deep insights, generous support, and constant encouragement. The author also would like to thank Takayuki Okuda, Temma Aoyama, and Tatsuro Hikawa for helpful advice and comments.

This work was supported by JST SPRING, Grant Number JPMJSP2108. 

\section{Preliminaries}
\subsection{Properness relation $\pf$ and equivalence relation $\sim$}
In this section, we review the equivalence relation $\sim$ introduced by Kobayashi~\cite{Kob96}.

\begin{dfn}\label{def:proper}(\cite{Kob96}) Let $L,H$ be subsets in a locally compact group $G$. 
\begin{enumerate}
    \item We write $L\pf H$ if the intersection $L\cap SHS^{-1}$ is relatively compact for any compact subset $S$ of $G$. Then, we say the pair $(L,H)$ is proper.
    \item We write $L \sim H$ if there exists a compact subset $S$ of $G$ such that $L \subset SH S^{-1}, H \subset SLS^{-1}$.
\end{enumerate}
\end{dfn}
We obtain the following proposition regarding $\pf$ and $\sim$. 
\begin{prop}(\cite{Kob96})\label{prop:simpf}
Let $H,L,L^\prime$ be subsets of a locally compact group $G$. 
\begin{enumerate}
\item If $L\sim L^\prime, L\pf H$, then we have $L^\prime \pf H$.
\item In particular, if the subsets $L,H$ are closed subgroups, then the following three conditions are equivalent:
\begin{itemize}
    \item $L\pf H$.
    \item The group action $L \curvearrowright G/H$ is proper.
    \item The group action $H\curvearrowright G/L $ is proper.
\end{itemize}
\end{enumerate}
\end{prop}
By Proposition~\ref{prop:simpf}~(1), classifying proper group actions is equivalent to classifying proper pairs.
Identifying subgroups up to this equivalence relation $\sim$ helps simplify the study of properness by Proposition~\ref{prop:simpf}~(2).

Moreover, this relation is known as the coarsest one that preserves properness. This is formulated by Kobayashi~\cite{Kob96} as follows:
\begin{fact} \emph{(Discontinuity duality theorem)}
    Let $L,L^\prime$ be subsets in a Lie group $G$. If the equation 
    \[ \pf(L,G) = \pf (L^\prime,G) \]
    holds, then we obtain $L\sim L^\prime$.
Here, we define the set $\pf(L,G)$ of subgroups of $G$ as follows:
\[ \pf (L,G) = \{ H \subset G \colon L \pf H \}.\]
\end{fact}

The duality theorem shows that the equivalence class of $L$ with respect to $\sim$ is completely determined by properness. It was proved in the case where $G$ is a reductive group by Kobayashi~\cite{Kob96}, and in the case where $G$ is a general Lie group by Yoshino\cite{Yos07a}. 

\subsection{The (CI) condition and properness}
In this subsection, we review the (CI) condition.
\begin{dfn}\label{def:CI}(\cite{Kob90})
    A pair $(L,H)$ of close subgroups of a Lie group $G$ is (CI) if the subgroup $L\cap gHg^{-1}$ is compact for all $g\in G$.
\end{dfn}
It immediately follows that a pair $(L,H)$ is (CI) if it is proper. The opposite implication is not true in general.

\section{The notations $\lesssim_+, \asymp_+,\lesssim_\times, \asymp_\times $}

In order to study asymptotic behavior, we define some notation and prove basic properties in this section.

\begin{dfn}
Let $a=(a_n)_{n\in \N}, b= (b_n)_{n\in \N}$ be sequences of real numbers.
\begin{enumerate}
\item We write $a\lesssim_+ b$ if there exists a positive number $C>0$ such that $b_n -a_n < C$ for sufficiently large $n$.
\item We write $a\asymp_+ b$ if $a \lesssim_+ b$ and $b\lesssim_+a$.
\item Assume $b$ is nonzero for sufficiently large $n$. We write $a \lesssim_\times b $ if there exists a constant $C>0$ such that $\frac{a_n}{b_n} < C$ for sufficiently large $n$.
\item We write $a\asymp_\times b $ if $a\lesssim_\times b$ and $b\lesssim_\times a$.
\item We write $a \gnsim_\times b$ if for any $R>0$ there exists $N$ such that $\frac{a_n}{b_n} > R$ for $n>N$.
\end{enumerate}
\end{dfn}

First, we prove some basic properties concerning the notation $\lesssim_\times$.

\begin{prop}
    Let $a=(a_n)_n,b=(b_n)_n,c=(c_n)_n,d=(d_n)_n$ be sequences of real numbers.
\begin{enumerate}
    \item If $a\lesssim_+ b$ and $b\lesssim_\times c$, then $a\lesssim_\times c$.
    \item If $a\lesssim_\times b$, then $b\asymp_\times a+b$.
    \item If $|a | \lnsim_\times |b|$, then $b\asymp_\times b-a$.
    \item If $a\lesssim_\times b$ and $c\lesssim_\times d$, then $a\cdot b\lesssim_\times c\cdot d$.
    \item If $a\lesssim_\times b$ and $c\lesssim_\times d$, then $a+ b\lesssim_\times c+ d$.
\end{enumerate}
\end{prop}

\begin{proof}
It is obvious, except for the statement (3).

Since $a \lnsim_\times b$, we have $a_n < \frac{1}{2}b_n$ for sufficiently large $n$. Thus, we have 
\[ \frac{b_n}{b_n - a_n} < \frac{b_n}{b_n - \frac{1}{2}b_n} = 2, \]
which implies $b \lesssim_\times b-a$.
\end{proof} 

\begin{rmk}
The notation $\lesssim_\times$ does not allow for the operation of `rearranging terms', that is, the evaluation $a\lesssim_\times b $ does not imply $a-c\lesssim_\times b-c$ for sequences $a, b, c$. For example, the sequences which are defined by $a_n = 2n , b_n = c_n =n$ satisfies $a\lesssim_\times b$, but does not satisfy $a-c \lesssim_\times b-c$.
\end{rmk}

Next, we state some properties concerning $\asymp_+$, and relationships between $\lesssim_+$ and $\lesssim_\times$. The proof is straightforward, and thus omitted.

\begin{prop}\label{prop:seqasymp}
\begin{enumerate}
\item If $a\lesssim_+ b, b\lesssim_+ c$, then $a\lesssim_+ c$.
\item If $a\lesssim_+b$, then $a+c\lesssim_+ b+c$.
\item If $a\lesssim_+ b, c\lesssim_+ d$, then $a+c\lesssim_+ b+d $.
\item If $a,b$ diverge to infinity as $n\to \infty$, then the condition $a\asymp_+b$ implies $a\asymp_\times b$.
\item $a\lesssim_+ b \iff \exp a \lesssim_\times \exp b $.
\item For all $\alpha >0 $, $a\lesssim_\times b \iff a^{\alpha} \lesssim_\times b^{\alpha}.$
\end{enumerate}

\end{prop}

We end this section with a proposition that will appear repeatedly in the sequel.
\begin{prop}\label{prop:normcpt}
    Let $G=GL_N(\R)$ and $S$ be a compact subset of $G$. Let $v = (v_n)_n \subset \R^N , s=(s_n)_n  \subset S$ be sequences. Then we have 
\begin{itemize}
    \item $||sv|| \asymp_\times||v||,$
    \item $||^tvs|| \asymp ||^tv||.$
\end{itemize}
\end{prop}
\begin{proof}
   We prove the first statement. The second statement follows similarly.

   For all $v_n \in \R$, there exists $k_n \in K\coloneqq O(N)$ such that $k_nv_n = \begin{pmatrix} ||v_n|| \\ 0 \\ \vdots \\ 0\end{pmatrix}$. Since the sequence $s_n k_n^{-1}$ is contained in the compact set $SK$, we may assume that $v_n = \begin{pmatrix} t_n \\ 0 \\ \vdots \\ 0 \end{pmatrix}$ for some $t_n \in \R^\mathbb{N}$. Here, we prove the following estimate.
   \begin{claim*}
       The first column $(s_{*1})_n \coloneqq \pmat{(s_{11})_n \\ \vdots \\ (s_{N1})_n}$ of the sequence $s_n$ satisfies $||(s_{*1})_n|| \asymp_\times 1.$
   \end{claim*}
    \begin{proof}
        Since $s$ is contained in the compact set $S$, the absolute value of each entry is bounded above by some positive number $M>0$, and thus the norm $||(s_{*1})_n||$ is bounded above. Moreover, since $|\det s_n|  \asymp_\times  1$, we have
    \begin{align*}
        ||(s_{*1})_n|| &\asymp_\times |(s_{11})_n| + \cdots + |(s_{N1})_n| \\
        &\geq \frac{|(s_{11})_n||(\tilde{s}_{11})_n|+ \cdots+ |(s_{N1})_n||(\tilde{s}_{N1})_n|}{M^{N-1}}\\
        &\geq \frac{|\det s_n|}{M^{N-1}} \asymp_\times 1,
    \end{align*}
    where $(\tilde{s}_{ij})_n$ denotes the cofactor of $(s_{ij})_n$. This is the desired lower estimate.
    \end{proof}
    By using the above claim, we have 
    \[ ||s_nv_n|| = |t_n| \cdot||(s_{*1})_n|| \asymp_\times |t_n| =||v_n||.\]
\end{proof}



\section{Classification of proper pairs in $G=\SR$}
The aim of this section is to prove Theorem~\ref{thm:properSR}. We begin by fixing a convention for writing elements of $\GR$ throughout this paper; we use the same convention for elements of $\SR$.
 
\begin{conv}
    Let $g=\begin{pmatrix} a & b \\ c & d\end{pmatrix} \in \GL$ and $v=\begin{pmatrix}e \\ f \end{pmatrix} \in \R^2$. As explained in the introduction, an element of the affine transformation group $\GR$ is written as
    \[
    \begin{pmatrix} g&v\end{pmatrix}
    = \begin{pmatrix} \begin{pmatrix} a & b \\ c & d\end{pmatrix}& \begin{pmatrix}e \\ f \end{pmatrix}\end{pmatrix}.
    \]
    In this paper, we write it instead as
    \[
    \begin{pmatrix} a & b & e\\ c & d & f \end{pmatrix}.
    \]
\end{conv}

\subsection{Classification of the Lie subalgebras of $\sr$}
In this subsection, we review the classification of Lie subalgebras of $\sr$. The classification is not stated explicitly in \cite{CKZ24}, but we can prove it in the same fashion, so the proof is omitted.

We write 
\[
e_1 =\begin{pmatrix}
    0 & 0 & 0 \\ 1 & 0 & 0
\end{pmatrix},
e_2=
\begin{pmatrix}
    1 & 0 & 0 \\ 0 & -1 & 0
\end{pmatrix},
e_3=
\begin{pmatrix}
    0 & -1 & 0 \\ 0 & 0& 0
\end{pmatrix},
f_1=
\begin{pmatrix}
    0 & 0 & 1\\ 0&0 & 0
\end{pmatrix},
f_2 =
\begin{pmatrix}
    0 & 0 & 0 \\ 0 & 0 &1
\end{pmatrix}.
\]
\begin{lem}\label{lem:clsr}
Up to conjugacy, the Lie subalgebras of $\sr$ are of the following forms:
\[
\sl,\quad \R^2,\quad \s=\spnR{e_2,f_1},\quad \l=\spnR{e_1,f_1},
\]
\[
\m=\spnR{e_1+f_1,\; f_2}\supset \n=\spnR{e_1+f_1},
\]
\[
\mathfrak{so}(2),\quad \spnR{e_1},\quad \spnR{e_2},\quad \spnR{e_2,e_1},
\]
\[
\R^2,\quad \mathfrak{so}(2)\oplus \R^2,\quad \spnR{f_1},
\]
\[
\spnR{e_2}\oplus \R^2,\quad \spnR{e_1}\oplus \R^2,\quad
\spnR{e_1,e_2,f_2},\quad \spnR{e_1,e_2,f_1,f_2}.
\]
Here we write $\mathfrak{so}(2)=\spnR{e_1+e_3}$ and $\R^2=\spnR{f_1,f_2}$.
\end{lem}

\subsection{Equivalent subgroups in the sense $\sim$}
In this subsection, we collect the connected subgroups stated in Lemma~\ref{lem:clsr} that are pairwise equivalent. Indeed, the following relations hold.
\begin{prop}\label{prop:simsubgrpofSR}
    \begin{enumerate}
        \item $\{e\} \sim SO(2).$
        \item $\R^2 \sim \R \sim SO(2)\ltimes \R^2.$
        \item $\SL \sim \{ \begin{pmatrix}e^a & 0 & 0 \\ 0 & e^{-a}& 0 \end{pmatrix}\colon a\in \R \} \sim \{ \begin{pmatrix}1 & 0 & 0 \\ b & 1& 0 \end{pmatrix}\colon a\in \R \}  $\\
        $\ \qquad \sim \{ \begin{pmatrix}e^a & 0 & 0 \\ b & e^{-a}& 0 \end{pmatrix}\colon a,b\in \R \}.$
        \item $\SL \ltimes \R^2 \sim \{ \begin{pmatrix}e^a & 0 & e \\ 0 & e^{-a}& f \end{pmatrix}\colon a,e,f\in \R \} $\\
        $\sim \{ \begin{pmatrix}e^a & 0 & 0 \\ b & e^{-a}& e \end{pmatrix}\colon a,b,e\in \R \} \sim \{ \begin{pmatrix}e^a & 0 & e \\ b & e^{-a}& f \end{pmatrix}\colon a,b,e,f\in \R \} $\\
        $\sim \{ \begin{pmatrix}1 & 0 & e \\ b & 1& f \end{pmatrix}\colon b,e,f\in \R \}.$      
    \end{enumerate}
\end{prop}
We obtain Theorem~\ref{thm:properSR}~(1) as a corollary of  Proposition~\ref{prop:simsubgrpofSR}.

This proposition immediately follows by the following fact.
\begin{fact}\label{fact:kakknk}
    \begin{enumerate}
        \item (KAK-decomposition, \cite[Thm.\  7.39]{Kna02})  Let $G$ be a reductive Lie group, $G=K\times \mathfrak{p}$ be a Cartan decomposition, and $\a$ be a maximal abelian subalgebra in $\mathfrak{p}$. Then, we have $G=K\exp (\a)K$.
        \item (Kostant decomposition, \cite{Kos73}) Let $G$ be a semisimple Lie group, and $G=KAN$ be an Iwasawa decomposition. Then, we have $G=KNK$.
    \end{enumerate}
\end{fact}

\subsection{Properness of the connected pairs}

In this subsection, we determine the properness of the pairs in $\SR$ up to the equivalence $\sim$. Our goal is Theorem~\ref{thm:properSR}~(2), and we review its statement.

\newtheorem*{thm:properSR}{\rm\bf Theorem~\ref{thm:properSR}}
\begin{thm:properSR}{\itshape (2)
The properness of pairs among $\SL,$ $\R^2,$  $S,$ $L,$ $M,$ and $N$ is summarized in Table~\ref{tab:properSR}. In Table~\ref{tab:properSR}, we mark an entry with $\pf$ if the corresponding pair is proper; otherwise the entry is left blank.}
\end{thm:properSR}
\begin{table}[hbtp]
 \captionsetup{list=no} 
\caption*{{\bf Table~\ref{tab:properSR}:} Properness of pairs of connected subgroups of $\SR$}\
\centering
\renewcommand{\arraystretch}{1.5}
\begin{tabular}{c@{\quad}cccccc}
\toprule
 \diagbox{$L$}{$H$} & $SL_2(\R)$ & $\R^2$ & $S$ & $L$ & $M$ & $N$ \\
\midrule
$SL_2(\R)$ &  & $\pf$ &  &  & $\pf$ & $\pf$ \\
$\R^2$ & $\pf$ &  &  &  &  & $\pf$ \\
$S$ &  &  &  &  &  & $\pf$ \\
$L$ &  &  &  &  &  &  \\
$M$ & $\pf$ &  &  &  &  &  \\
$N$ & $\pf$ & $\pf$ & $\pf$ &  &  &  \\
\bottomrule
\end{tabular}
\end{table}

\begin{rmk}
We do not refer the ambient group $G$ explicitly, as the properness described above is not changed whether $G=\GR$ or $G=\SR$. We now prove Theorem~\ref{thm:properSR}~(2) under the assumption that $G=\GR$, but a parallel discussion works under the assumption that $G= \SR$.
\end{rmk}

By using the following lemma, we should only prove $\SL \pf N, S \pf N$ and $L \not\pf N$. The first statement immediately follows by the definition, and the second one is  proved in \cite[Prop.\ A.2.1]{Kob90}.
\begin{lem}
    \begin{enumerate}
        \item If a pair of subsets in a topological group has non-compact intersection, then the pair is non-proper.
        \item The connected subgroups whose pair with $\SL$ is proper are $\R^2,$ $M$ and $N$.
        \end{enumerate}
\end{lem}

We prove the remaining part of Theorem~\ref{thm:properSR}~(1).
\begin{lem}
    \begin{enumerate}
        \item $\R^2 \pf N$.
        \item $S \pf N$.
        \item $L \not\pf N$.
    \end{enumerate}
\end{lem}
\begin{proof}
    (1) We prove the action $N \curvearrowright \GR/\R^2 \cong \GL$ is proper. Consider the continuous group isomorphism
    \[ \Lin\colon N \to \Lin(N)\coloneqq \{\begin{pmatrix} 1 & 0 \\ t & 1\end{pmatrix}\colon t \in \R\},
    \begin{pmatrix}  1 & 0 & t \\ t& 1 & \frac12 t^2 \end{pmatrix}
    \mapsto \begin{pmatrix} 1 & 0 \\ t & 1\end{pmatrix}.\]
    The isomorphism $\Lin$ is compatible to the group actions, that is, the following commutative diagram holds for all $s\in N$: 
    \[
    \begin{tikzcd}
  \GR/\R^2 \arrow[r, "\cong"] \arrow[d, "s \cdot" left]
    & \GL \arrow[d, "\Lin(s) \cdot" left] \\
  \GR/\R^2 \arrow[r, "\cong"]
    & \GL
\arrow[phantom, from=1-1, to=2-2, "\circlearrowright" description, pos=0.52]
   \end{tikzcd}
   \]
Since $\Lin(N)$ is a closed subgroup of $\GL$, the action $\Lin(N) \curvearrowright \GL$ is proper. Thus, the action $N \curvearrowright \GR/\R^2$ is also proper.

    (2) Assume the contrary. There exist sequences $\begin{pmatrix} 1 & 0 & t_n \\ t_n & 1 &\frac12 t_n^2\end{pmatrix} \in N$, and $\begin{pmatrix} a_n & c_n & e_n \\ b_n & d_n &f_n\end{pmatrix},$ $\begin{pmatrix} A_n & C_n & E_n \\ B_n & D_n & F_n\end{pmatrix} $ which are contained in some compact set of $\GR$, such that the product $\begin{pmatrix} a_n & c_n & e_n \\ b_n & d_n &f_n\end{pmatrix} \begin{pmatrix} 1 & 0 & t_n \\ t_n & 1 &\frac12 t_n^2\end{pmatrix} \begin{pmatrix} A_n & C_n & E_n \\ B_n & D_n & F_n\end{pmatrix} \in S$ and $|t_n| \to \infty$. We have 
    \begin{align*}
    &\begin{pmatrix} a_n & c_n & e_n \\ b_n & d_n &f_n\end{pmatrix} \begin{pmatrix} 1 & 0 & t_n \\ t_n & 1 &\frac12 t_n^2\end{pmatrix} \begin{pmatrix} A_n & C_n & E_n \\ B_n & D_n & F_n\end{pmatrix} \\
    &= \begin{pmatrix} * & C_n(a_n+c_nt_n) +c_nD_n & * \\ A_n(b_n+d_nt_n)+d_nB_n & * & (b_n+d_nt_n)E_n+d_nF_n+b_nt_n+\frac12d_nt_n^2+f_n \end{pmatrix}
    \end{align*}
    Since (2,3)-entry is zero, we have 
    \[ |d_n| \leq \frac{2}{t_n^2}(|b_n| +|d_n||E_n||t_n|+ |b_n||t_n| +|f_n|) \lesssim_\times \frac{1}{|t_n|},\]
    and 
    \[ |d_nt_n+b_n| \geq |b_n| - |2(\frac12 d_nt_n+b_n)| \asymp_\times 1.\]
    The last evaluation follows by the following evaluations;
    \[ |b_n| \geq \frac{\begin{vmatrix} a_n & b_n \\ c_n &d_n \end{vmatrix}-|a_n||d_n|}{|c_n|} \gtrsim_\times 1,\]
    \[ |\frac12 d_nt_n+b_n| \leq \frac{1}{|t_n|}(|b_n|+|d_n||E_n||t_n| +|d_n||F_n| +|f_n|) \lesssim_\times \frac{1}{|t_n|}.\]
    Similarly, we have $|c_n| \asymp_\times 1$. Since (1,2) entry is zero, we have 
    \[ |C_n| \leq \frac{|c_n||D_n|}{|c_n||t_n|-|a_n|} \lesssim_\times \frac{1}{|t_n|}.\]
    Thus, we have $|A_n| \asymp_\times 1$. Hence, the (2,1) entry
    \[ |A_n(b_nt_n+d_n) +d_nB_n| \geq |A_n||b_nt_n+d_n| -|d_n||B_n| \gtrsim_\times |t_n|\]
    is non-zero, which contradicts that the (2,1)-entry is zero. Therefore, we have $S \pf N$.

    (3) For $|t| >1$, we have 
    \[ \begin{pmatrix} 1 & -\frac{2}{t} & 0 \\ 0 & 1 & 0 \end{pmatrix}\begin{pmatrix} 1 & 0 & t \\ t & 1& \frac12 t^2 \end{pmatrix}\begin{pmatrix} -1 & \frac2t & 0 \\ 0 & -1& 0 \end{pmatrix}=\begin{pmatrix} 1 & 0 & 0 \\ -t & 1& \frac12 t^2 \end{pmatrix} \in L .\]
    Since all entries of the matrices $\begin{pmatrix} 1 & -\frac{2}{t} & 0 \\ 0 & 1 & 0 \end{pmatrix}$ and $\begin{pmatrix} -1 & \frac2t & 0 \\ 0 & -1& 0 \end{pmatrix} \in \GR$ is bounded, and $\begin{pmatrix} 1 & 0 & t \\ t & 1& \frac12 t^2 \end{pmatrix} \in N$. Therefore, we have $L \not\pf N.$
\end{proof}

\subsection{Non-equivalence of the subgroups with respect to the relation $\sim$ }\label{subsec:ne}
We finish this section by proving the following proposition.
\begin{prop}
The subgroups of $\SR$, $\{e\}$, $\SR$, $\SL$, $\R^2$, $S$, $L$, $M$, $N$ are pairwise non-equivalent in the sense of the notation $\sim$.
\end{prop}
By Proposition~\ref{prop:simpf} and Theorem~\ref{thm:properSR}~(2), it suffices to prove the following lemma.
\begin{lem}
    $L \not\sim \SR.$
\end{lem}
\begin{proof}
    We prove $S^\prime \pf L$, where $S^\prime$ is the non-compact subset in $\SR$ defined by
    \[ S^\prime \coloneqq \{ \begin{pmatrix} 1 & 0 & t^2 \\ t & 1 & 0 \end{pmatrix}\colon t\in \R\}. \]
    Assume the contrary, then there exists sequences $\begin{pmatrix} a_n & c_n & e_n \\ b_n & d_n & f_n \end{pmatrix}, \begin{pmatrix} A_n & C_n & E_n\\ B_n & D_n& F_n \end{pmatrix}$ which is contained in some compact set, such that
    \begin{align*} &\begin{pmatrix} a_n & c_n & e_n \\ b_n & d_n & f_n \end{pmatrix} \begin{pmatrix} 1 & 0 & t_n^2 \\ t_n & 1 & 0 \end{pmatrix} \begin{pmatrix} A_n & C_n & E_n\\ B_n & D_n& F_n \end{pmatrix} \\
    &= \begin{pmatrix} (a_n+c_nt_n)A_n+c_nB_n & (a_n +c_nt_n)C_n+c_nD_n & (a_n+c_nt_n)E_n+c_nF_n +a_nt_n^2+e_n \\ * & * & *\end{pmatrix} \in L
    \end{align*}
    Since the (1,3)-entry is zero, we have 
    \[ |a_n| \leq \frac{1}{|t_n|^2} (|e_n| + |c_n||F_n|+|a_n||E_n|+|c_n||E_n||t_n|) \lesssim_\times \frac{1}{|t_n|}.\]
    Thus, we have $|c_n| \asymp_\times1.$ Since (1,2)-entry is zero, we have 
    \[ |C_n| \leq \frac{|c_n||D_n|}{|c_n||t_n|-|a_n|} \lesssim_\times \frac{1}{|t_n|}.\]
    Thus, we have $|A_n| \asymp_\times 1.$ Hence, we have
    \[ |(a_n+c_nt_n)A_n+c_nB_n| \gtrsim_\times|c_n||A_n||t_n|-|a_n||A_n|-|c_n||B_n| \gtrsim_\times |t_n|,\]
    which contradicts that the (1,1)-entry is 1. Therefore, we have $S^\prime \pf L$.

    Obviously, $S^\prime \not\pf (\SR)$, thus we have $L \not\sim \SR$ by Proposition~\ref{prop:simpf}. 
    
\end{proof}


\section{Classification of proper pairs in $G=\GL$}

\subsection{Classification of Lie subalgebras of $\mathfrak{gl}_2(\R)$}
In this section, we review the classification of Lie subalgebras of $\mathfrak{gl}_2(\R)$ (for example, see \cite{PW77}).

\begin{fact}
Up to conjugacy, the Lie subalgebras of $\mathfrak{gl}_2(\R)$ are of the following forms:
\begin{gather*}
 \hfir= \R\begin{pmatrix} 1 & 0 \\ 0 & 1\end{pmatrix}, \hsec{\alpha} = \R\begin{pmatrix} \alpha + 1 & 0 \\ 0 & \alpha - 1\end{pmatrix}, \hthi = \R\begin{pmatrix} 1 & 0 \\ 1 & 1\end{pmatrix}, \\
 \hfou{\beta}= \R\begin{pmatrix} \beta + 1 & 0 \\ 0 & \beta - 1\end{pmatrix} \oplus \R \begin{pmatrix} 0 & 0 \\ 1 & 0\end{pmatrix}, 
 \hfif= \R\begin{pmatrix} 1 & 0 \\ 0 & 1\end{pmatrix}\oplus \R\begin{pmatrix} 0 & 0 \\ 1 & 0\end{pmatrix},\\
 \hsix= \R\begin{pmatrix} 1 & 0 \\ 0 & -1\end{pmatrix}\oplus \R\begin{pmatrix} 1 & 0 \\ 0 & 1\end{pmatrix}, \mathfrak{o}^\prime(\gamma)= \R\begin{pmatrix} \gamma & -1 \\ 1 & \gamma\end{pmatrix},\mathfrak{co} =\R \begin{pmatrix} 0 & -1 \\ 1 & 0\end{pmatrix}\oplus\R \begin{pmatrix}1 & 0 \\ 0 & 1\end{pmatrix},\\ \mathfrak{p}=\R \begin{pmatrix} 0 & 0 \\ 1 & 0\end{pmatrix}\oplus \R \begin{pmatrix} 1 & 0 \\ 0 & -1\end{pmatrix} \oplus \R \begin{pmatrix} 1 & 0 \\ 0 & 1\end{pmatrix},\mathfrak{u}^\prime= \R\begin{pmatrix} 0 & 0 \\ 1 & 0\end{pmatrix},\sl. 
\end{gather*}
\end{fact}
We denote by the corresponding capital letter the analytic subgroup associated with each Lie subalgebra.

\subsection{Equivalent subgroups with respect to the relation $\sim$ }
In this subsection, our goal is the following proposition. It directly follows from KAK-decomposition or Kostant-decomposition (see Fact~\ref{fact:kakknk}).
\begin{prop}
    We have the following equivalences:
    \begin{enumerate}
        \item $\GL \sim \Hfif \sim \Hsix \sim P.$
        \item $\{e\} \sim O^\prime(0).$
        \item $\Hfir \sim O^\prime(\gamma) \sim CO$ if $\gamma\neq0$.
        \item $\SL \sim U^\prime.$
    \end{enumerate}
\end{prop}

\subsection{Review of the properness criterion for $GL_2(\R)$}
We now review the properness criterion in case $G=GL_2(\R)$. First, we prepare some notions to state it.

\begin{dfn}
\begin{enumerate}
\item The space $\a_+\coloneqq \{ (x,y) \in \R^2 | x\geq y \}$ with the metric induced by $\R^2$.
\item The Cartan projection $\mu \colon G\to \a_+$ is defined by 
\[ \mu(g) = \left( \frac{1}{2}\log(\sigma_1(g)),\frac{1}{2} \log(\sigma_2(g))\right) ,\]
where the number $\sigma_i(g)$ are the eigenvalues of the positive-definite symmetric matrix $^tgg$, and satisfies $\sigma_1(g)>\sigma_2(g)$.
\end{enumerate}
\end{dfn}
We review the properness criterion in case where $G=\GL$.
\begin{fact}\label{fact:pc}(Properness criterion, \cite{Kob89,Kob96,Ben96}) Suppose that $H,L,L^{\prime}$ are subsets of $\GL$.
\begin{enumerate}
    \item  \begin{align*}
        H \pf L \ in \ G &\iff \mu(H) \pf \mu(L) \ \text{in} \ \a \\
        &\iff \mu(H) \cap \bar{B}_r(\mu(L)) \text{ is  bounded}\ \text{for all $r>0$  in} \ \a_+
    \end{align*}
    \item \begin{align*}
        L\sim L^{\prime} \ in \ G &\iff \mu(L) \sim \mu(L^{\prime}) \  in \ \a ,\\
        & \iff \mu(L) \subset \bar{B}_R(\mu(L^\prime)), \mu(L^\prime) \subset \bar{B}_R(\mu(L)) \\
        & \qquad\quad \ in \ \a_+ \ for \ some\  R>0.
        \end{align*}
\end{enumerate}
Here, the closed $R$-neighborhood $\bar{B}_R(S)$ is defined by $\bar{B}_R(S) \coloneqq \{ x\in \a_+ \colon d(x,S) \leq R\}$.
\end{fact}

In what follows, we shall use the properness criterion in the form of the following fact.
\begin{fact}\label{fact:ue}(\cite[Prop.\ 5.1]{Ben96})
    For any compact set $C$ in $\GL$, there exists $R>0$ such that $\mu(CgC) \subset \bar{B}_R(\mu(g))$.
\end{fact}

\subsection{Properness of the connected pairs}
In this subsection, we prove Theorem~\ref{thm:properGL}~(2)

\newtheorem*{thm:properGL}{\rm\bf Theorem~\ref{thm:properGL}}
\begin{thm:properGL}{\itshape(2)
 The properness of the pairs among the connected subgroups $Z,U,A(\alpha),B(\alpha)$ is described in the following Table~\ref{tab:properGL}. Here, each entry in the table that specifies a condition on the parameters $\alpha ,\beta$ provides a necessary and sufficient condition for the corresponding pair to be proper.}
\end{thm:properGL}
\begin{table}[htbp]
\centering
\captionsetup{list=no} 
\caption*{{\bf Table~\ref{tab:properGL}:} Properness of the pairs in $\GL$}
\begin{tabular}{c@{\quad}cccc}
    \toprule
     \diagbox{$L$}{$H$} &  $Z$ & $A(\beta)$ & $U$ & $B(\beta)$  \\
    \midrule
    $Z$ & & $\pf$ & $\pf$ & $\pf$   \\
    $A(\alpha)$ &$\pf$ & $|\alpha| \neq |\beta|$ & $\pf$ & $|\alpha|>|\beta|$  \\
    $U $  &$\pf$ & $\pf$ &  & $\pf$   \\
    $B(\alpha)$  & $\pf$ & $|\alpha|<|\beta|$ & $\pf$ &   \\
    \bottomrule
    \end{tabular}
    \end{table}   

The Cartan projection of $\Hfir, \Hsec\gamma, \GL$ can be calculated easily. For analyzing the Cartan projection $\Hthi,\Hfou\gamma$, we prepare the following lemma.
\begin{lem}\label{lem:muh3h4}
    \begin{enumerate}
        \item Let $(x,y) \in \mu(\Hthi)$. For fixed $\gamma\in \R_{\geq 0}$, we have 
        \[ \frac{\gamma-1}{\gamma+1}x <y \ \text{and} \ y+\log |y|<x  \quad \text{for} \ x>>0,\]
        \[ \frac{\gamma+1}{\gamma-1}x <y \ \text{and} \ y<x-\log |x|  \quad \text{for} \ x<<0.\]
        \item $\mu(\Hfou\gamma)= \{ (x,y) \in \a_+\colon y \leq \frac{\gamma -1}{\gamma+1}x, \quad y\leq \frac{\gamma+ 1}{\gamma-1}x\}$
    \end{enumerate}
\end{lem}
\begin{proof}
(1) By the definition of Cartan projection, we have
\[ \mu \begin{pmatrix} e^t & 0 \\ te^t & e^t\end{pmatrix}= \begin{pmatrix} t+\frac12\log \left( \frac{t^2+2+\sqrt{t^4+4t^2}}{2}\right)\\
t+\frac12\log \left( \frac{t^2+2-\sqrt{t^4+4t^2}}{2}\right)\end{pmatrix}.\]
Suppose $|t|>1.$ Then, we have
\begin{align*}
    y&= t+\frac12 \log \frac{2}{t^2+2+\sqrt{t^4+4t^2}}\\
    &< t+ \frac12 \log \frac{2}{2t^2}\\
    &= t-\log|t|  <t ,
\end{align*}
and 
\begin{align*}
    x&> t + \frac12 \log (t^2) = t+ \log |t|.
\end{align*}
Thus, we have 
\[ 
x \geq   y+\log|y|  \quad \text{if } t > 1.
 \]
In the same fashion, we obtain $y\leq x - \log|x|$ when $t<-1$.
Therefore, we obtain the upper estimate of $\mu(H_3).$

On the other hand, we have for any $\varepsilon >0$,
\begin{align*} 
y&\geq t + \frac12 \log \frac{1}{2t^2}\\
&= t - \frac12 \log 2- \log|t| \geq (1-\varepsilon)t,
\end{align*}
and 
\[    x \leq (1+\varepsilon) t \] 
if $t>0$ is sufficiently large. Thus, we have 
\[ y \geq \frac{1-\varepsilon}{1+\varepsilon}x \]
for sufficient large $x>0$. Similarly, we have 
\[ y \geq \frac{1+\varepsilon}{1-\varepsilon}x\]
for sufficiently small $x<0$. Therefore, we obtain a lower estimate of $\mu(\Hthi)$.

(2) We write $g_{u,t}=\begin{pmatrix} e^{(\gamma+1)t} & 0 \\ u &e^{(\gamma -1)t} \end{pmatrix}$. By a calculation, we have 
\[\mu(g_{u,t})= \begin{pmatrix} \sigma_1(g_{u,t})\\ \sigma_2(g_{u,t})\end{pmatrix}= \begin{pmatrix} \frac12\log (\frac12 (A+\sqrt{A^2-4 e^{4\gamma t}}))\\ \frac12\log (\frac12 (A-\sqrt{A^2-4 e^{4\gamma t}}))\end{pmatrix},\]
where we denote $A= e^{2(\gamma +1)t}+e^{2(\gamma -1)t}+u^2$. It is on the piecewise linear line $\mu(H_2^\gamma)$ when $u=0$ (see Figure~\ref{fig:cartanGL}). 

Fix $t\in \R$. Since we have 
\[ \sigma_1(g_{u,t})+\sigma_2(g_{u,t}) = 2\gamma t,\]
the image of $\mu(g_{u,t})$ is on the line $x+y=2\gamma t$. As the parameter $|u|$ increases continuously, the first coordinate increases continuously. Thus, the image of $\mu(g_{u,t})$ is the ray starting at the point $((\gamma+1)t,(\gamma-1)t)$. Therefore, we obtain the Cartan projection as the desired form.
\end{proof}

\begin{proof}[Proof of Thoerem~\ref{thm:properGL}~(2).]
We illustrate the images of Cartan projection of $\Hfir, \Hsec\gamma, \Hthi, \Hfou\gamma$ in Figure~\ref{fig:cartanGL}.
\begin{figure}[htbp]
\centering
\begin{tikzpicture}
  \begin{axis}[
    width=10cm, height=10cm,
    axis lines=middle,
    xlabel={$x$}, ylabel={$y$},
    xmin=-6, xmax=6, ymin=-6, ymax=6,
    ticks=none,
    clip=true,
    samples=2,
    line cap=round, line join=round,
]

\addplot[name path=top,  draw=none, domain=-6:6, samples=2] {6};
\addplot[name path=diag, draw=none, domain=-6:6, samples=2] {x};
\addplot[
  pattern= north west lines,
  pattern color=black!50,
  draw=none,
] fill between[of=top and diag, soft clip={domain=-6:6}];

    \addplot[very thick, black, domain=-6:6,samples=2] {x};

\addplot[name path=bottom, draw=none, domain=-6:6, samples=2] {-6};

\path[name path=minline]
  (axis cs:-6,-12) -- (axis cs:0,0) -- (axis cs:6,3);

\addplot[gray!40, fill opacity=0.20, draw=none,samples=2]
  fill between[of=bottom and minline];

\addplot[very thick, domain=0:6,  samples=2] {0.5*x};
\addplot[very thick, domain=-6:0, samples=2] {2*x};

    
\addplot[very thick, black, domain=-6:-2,  samples=60,  parametric]
  ({ x + 0.5*ln(0.5*(x^2 + 2 + sqrt(x^4 + x^2))) },
   { x + 0.5*ln(0.5*(x^2 + 2 - sqrt(x^4 + x^2))) });

\addplot[very thick, black, domain=-2:2,   samples=40,  parametric]
  ({ x + 0.5*ln(0.5*(x^2 + 2 + sqrt(x^4 + x^2))) },
   { x + 0.5*ln(0.5*(x^2 + 2 - sqrt(x^4 + x^2))) });

\addplot[very thick, black, domain=2:6,    samples=60,  parametric]
  ({ x + 0.5*ln(0.5*(x^2 + 2 + sqrt(x^4 + x^2))) },
   { x + 0.5*ln(0.5*(x^2 + 2 - sqrt(x^4 + x^2))) });


\node[inner sep=1.2pt] (H1)
  at (axis cs:3.0 ,4.0) {$\mu(\Hfir)$};
\draw[-{Latex[length=2.2mm]}, thick]
  (H1) -- (axis cs:4.2,4.2); 

\node[inner sep=1.2pt] (H2)
  at (axis cs:3,1) {$\mu(\Hsec{\gamma})$};
\draw[-{Latex[length=2.2mm]}, thick]
  (H2) -- (axis cs:4.0,2.0); 

\pgfmathsetmacro{\tmark}{-3.2}
\pgfmathsetmacro{\tTwo}{\tmark*\tmark}              
\pgfmathsetmacro{\root}{sqrt(\tTwo*(\tTwo+1))}      

\pgfmathsetmacro{\Xmark}{\tmark + 0.5*ln( 0.5*(\tTwo + 2 + \root) )}
\pgfmathsetmacro{\Ymark}{\tmark + 0.5*ln( 0.5*(\tTwo + 2 - \root) )}

\pgfmathsetmacro{\Xlab}{\Xmark - 1.5}
\pgfmathsetmacro{\Ylab}{\Ymark - 1.0}

\node[inner sep=1.2pt] (H3)
  at (axis cs:\Xlab,\Ylab) {$\mu(\Hthi)$};
\draw[-{Latex[length=2.2mm]}, thick]
  (H3) -- (axis cs:\Xmark,\Ymark);

    \node[anchor=west, inner sep=1.2pt]
      at (axis cs:3.8,-2.0) {$\mu(\Hfou{\gamma})$};

  \end{axis}
\end{tikzpicture}
  \caption{the Cartan projection of $\Hfir, \Hsec\gamma, \Hthi,\Hfou\gamma$}
  \label{fig:cartanGL}
\end{figure}
Note that the image $\mu(\Hfou\gamma)$ contains its boundary. By the figure, we have $\mu(H_1) \pf \mu(\Hsec\gamma), \mu(\Hfir) \pf \mu(\Hfou\gamma)$. Moreover, $\mu(\Hsec\alpha) \pf \mu(\Hsec\beta)$ if and only if $|\alpha|\neq |\beta|$, $\mu(\Hsec\alpha) \pf \mu(\Hfou\beta)$ if and only if $|\alpha|>|\beta|$, and $\mu(\Hfou{\gamma_1}) \pf \mu(\Hfou{\gamma_2})$ for any $\gamma_1,\gamma_2$.

Since a neighborhood of the line $x=y$ cannot contain the curves $x=y+\log|y|$ nor $x- \log|x|= y$, we have $\mu(\Hfir)\pf \mu(\Hthi)$ by Lemma~\ref{lem:muh3h4}.

On the other hand, $\mu(\Hthi)$ can be estimated lower by $y=\frac{\gamma^\prime -1}{\gamma^\prime+1}x \quad (x>>1)$ or $y=\frac{\gamma^\prime+1}{\gamma^\prime-1}x \quad (x<<-1)$. Since their neighborhood cannot contain an unbounded intersection with $\mu(\Hsec\gamma)$ or $\mu(\Hfou\gamma)$ if $\gamma^\prime > |\gamma|$, we have $\mu(\Hthi) \pf \mu(\Hsec\gamma)$ and $\mu(\Hthi)\pf \mu(\Hfou\gamma)$ for any $\gamma$.
\end{proof}
\begin{rmk}
    By Proposition~\ref{prop:simpf}, the connected subgroups $\Hfir, \Hsec\gamma, \Hthi, \Hfou\gamma$ are pairwise non-equivalent in the sense of the notation $\sim$.
\end{rmk}

\section{Classification of connected subgroups of $\GR$ up to $\sim$}\label{sec:clGR}

\subsection{Classification up to conjugacy}
In this subsection, we review the classification of subalgebras of $\mathfrak{gl}_2(\R)\ltimes\R^2$ proved by Chapovskyi-Koval-Zhur~\cite[Thm.\ 11]{CKZ24}.

We write 
\[
e_1 =\begin{pmatrix}
    0 & 0 & 0 \\ 1 & 0 & 0
\end{pmatrix},
e_2=
\begin{pmatrix}
    1 & 0 & 0 \\ 0 & -1 & 0
\end{pmatrix},
e_3=
\begin{pmatrix}
    0 & -1 & 0 \\ 0 & 0& 0
\end{pmatrix},
e_4=
\begin{pmatrix}
    1 & 0 & 0 \\ 0 & 1& 0
\end{pmatrix},
\]
\[
f_1=
\begin{pmatrix}
    0 & 0 & 1\\ 0&0 & 0
\end{pmatrix},
f_2 =
\begin{pmatrix}
    0 & 0 & 0 \\ 0 & 0 &1
\end{pmatrix}.
\]

\begin{fact}
A complete list of pairwise non-conjugate subalgebras of $\mathfrak{gl}_2(\R)\ltimes \R^2$
is given as follows. Unless otherwise specified, $\gamma$ denotes a real constant.
\begin{multicols}{2}
\begin{itemize}
  \item The case where the linear part is $\hfir$:
  \[
  \begin{array}{l}
    \l(\hfir,1)=\spnR{e_4},\\
    \l(\hfir,2)=\spnR{e_4,f_1},\\
    \l(\hfir,3)=\spnR{e_4,f_1,f_2}.
  \end{array}
  \]

  \item The case where the linear part is $\hsec{\gamma}$:
  \[
  \begin{array}{l}
    \l(\hsec{\gamma},1)=\spnR{e_2+\gamma e_4}\quad (\gamma\ge 0),\\
    \l(\hsec{1},2)=\spnR{e_2+e_4+f_2},\\
    \l(\hsec{\gamma},3)=\spnR{e_2+\gamma e_4,\; f_1},\\
    \l(\hsec{1},4)=\spnR{e_2+e_4+f_2,\; f_1},\\
    \l(\hsec{\gamma},5)=\spnR{e_2+\gamma e_4,\; f_1,\; f_2}\quad (\gamma\ge 0).
  \end{array}
  \]

  \item The case where the linear part is $\hthi$:
  \[
  \begin{array}{l}
    \l(\hthi,1)=\spnR{e_1+e_4},\\
    \l(\hthi,2)=\spnR{e_1+e_4,\; f_2},\\
    \l(\hthi,3)=\spnR{e_1+e_4,\; f_1,\; f_2}.
  \end{array}
  \]

  \item The case where the linear part is $\hfou{\gamma}$:
  \[
  \begin{array}{l}
    \l(\hfou{\gamma},1)=\spnR{e_2+\gamma e_4,\; e_1},\\
    \l(\hfou{1},2)=\spnR{e_2+e_4+f_2,\; e_1},\\
    \l(\hfou{-3},3)=\spnR{e_2-3e_4,\; e_1+f_1},\\
    \l(\hfou{\gamma},4)=\spnR{e_2+\gamma e_4,\; e_1,\; f_2},\\
    \l(\hfou{-1},5)=\spnR{e_2-e_4+f_1,\; e_1,\; f_2},\\
    \l(\hfou{-3},6)=\spnR{e_2-3e_4,\; e_1+f_1,\; f_2},\\
    \l(\hfou{\gamma},7)=\spnR{e_2+\gamma e_4,\; e_1,\; f_1,\; f_2}.
  \end{array}
  \]

  \item The case where the linear part is $\hfif$:
  \[
  \begin{array}{l}
    \l(\hfif,1)=\spnR{e_1,\; e_4,\; f_2},\\
    \l(\hfif,2)=\spnR{e_1,\; e_4},\\
    \l(\hfif,3)=\spnR{e_1,\; e_4,\; f_1,\; f_2}.
  \end{array}
  \]

  \item The case where the linear part is $\hsix$:
  \[
  \begin{array}{l}
    \l(\hsix,1)=\spnR{e_2,\; e_4,\; f_1},\\
    \l(\hsix,2)=\spnR{e_2,\; e_4},\\
    \l(\hsix,3)=\spnR{e_2,\; e_4,\; f_1,\; f_2}.
  \end{array}
  \]

  \item The case where the linear part is $\mathfrak{o}^\prime(\gamma)$:
  \[
  \begin{array}{l}
    \l(\mathfrak{o}^\prime(\gamma),1)=\spnR{e_1+e_3+\gamma e_4},\\
    \l(\mathfrak{o}^\prime(\gamma),2)=\spnR{e_1+e_3+\gamma e_4,\; f_1,\; f_2}.
  \end{array}
  \]

  \item The case where the linear part is $\mathfrak{co}$:
  \[
  \begin{array}{l}
    \l(\mathfrak{co},1)=\spnR{e_1+e_3,\; e_4},\\
    \l(\mathfrak{co},2)=\spnR{e_1+e_3,\; e_4,\; f_1,\; f_2}.
  \end{array}
  \]

  \item The case where the linear part is $\mathfrak{p}$:
  \[
  \begin{array}{l}
    \l(\mathfrak{p},1)=\spnR{e_1,\; e_2,\; e_4},\\
    \l(\mathfrak{p},2)=\spnR{e_1,\; e_2,\; e_4,\; f_1,\; f_2}.
  \end{array}
  \]

  \item The case where the linear part is $\mathfrak{u}^\prime$:
  \[
  \begin{array}{l}
    \l(\mathfrak{u}^\prime,1)=\spnR{e_1},\\
    \l(\mathfrak{u}^\prime,2)=\spnR{e_1+f_1}\quad(=\n),\\
    \l(\mathfrak{u}^\prime,3)=\spnR{e_1,\; f_2}\quad(=\l),\\
    \l(\mathfrak{u}^\prime,4)=\spnR{e_1+f_1,\; f_2}\quad(=\m),\\
    \l(\mathfrak{u}^\prime,5)=\spnR{e_1,\; f_1,\; f_2}.
  \end{array}
  \]

  \item The case where the linear part is ${\bf 0}$:
  \[
  \begin{array}{l}
    \l({\bf 0},1)=\spnR{f_1},\\
    \l({\bf 0},2)=\spnR{f_1,\; f_2}.
  \end{array}
  \]

  \item The case where the linear part is $\sl$:
  \[
    \sl,\quad \sr
  \]

  \item The case where the linear part is $\mathfrak{gl}_2(\R)$:
  \[
    \mathfrak{gl}_2(\R),\quad \mathfrak{gl}_2(\R)\ltimes \R^2.
  \]
\end{itemize}
\end{multicols}
Here, the subalgebras $\l,\m,\n$ are defined in Section~4.
\end{fact}

\subsection{Equivalent connected subgroups with respect to the relation $\sim$ }
In this subsection, our goal is the following proposition. It immediately follows by KAK-decomposition or Kostant-decomposition. 

\begin{prop}\label{prop:simsubgrp}
    The following equivalences hold in the sense of the relation~$\sim$.
    \begin{enumerate}
        \item $L(\Hfir,1) \sim L(O^\prime(\gamma),1) \sim L(CO,1)$ if $\gamma \neq 0$.
        \item $L(\Hfir,2) \sim L(\Hfir,3) \sim L(O^\prime(\gamma),2) \sim L(CO,2)$ if $\gamma \neq 0$.
        \item $\GL \sim L(\Hfif,2) \sim L(\Hsix,2) \sim L(P,1)$.
        \item $\GR \sim L(\Hfif,3) \sim L(\Hsix,3) \sim L(P,2).$
    \end{enumerate}
\end{prop}

\begin{rmk}
By Proposition~\ref{prop:simsubgrp}, the case where the linear part is $O^\prime(\gamma),CO,\,P$ reduced to the case where the linear part is $Z,\GL$. Thus, we do not have to care of such cases to consider properness.
\end{rmk}


\section{Proof of Theorems~\ref{prop:LinL=DB'}-\ref{thm:RCI}}
In this section, we discuss the easier case which can be proved independent of the classification of connected subgroups of $\GR$.

\subsection{Proof of Theorem~\ref{prop:LinL=DB'}}
In this subsection, we classify the connected subgroups when either of the linear part is $\Hfif,\Hsix$. First, we review the classification the connected subgroups $H$ acting properly on $\R^2$.

\begin{fact}(\cite[Prop.\ A.2.1]{Kob90})
    Up to equivalence with respect to the relation $\sim$ , the connected subgroups acting properly on $\R^2$ are of the following forms.
    \[  \R^2, M=\{ \begin{pmatrix} 1 & 0 & t \\ t & 1 & u\end{pmatrix}\colon t,u \in \R \}, N = \{ \begin{pmatrix} 1 & 0 & t \\ t & 1 & \frac12 t^2\end{pmatrix}\colon t \in \R\},\]
    \[ L(\Hsec1,2) = \{ \begin{pmatrix} e^t & 0 & 0 \\ 0 & 1 & t\end{pmatrix}\colon t \in \R\}\subset L(\Hsec1,5) = \{  \begin{pmatrix} e^t & 0 & u \\ 0 & 1 & t\end{pmatrix}\colon u,t \in \R \}.\]
\end{fact}

In the remainder of this subsection, we prove the following theorem.
\newtheorem*{prop:LinL=DB'}{\rm\bf Theorem~\ref{prop:LinL=DB'}}
\begin{prop:LinL=DB'}{\itshape
   Let $L$ and $H$ be connected closed subgroups of $\GR$. If $L \pf H$ in $G$, then $(L,H)$ is of the following form:
\[ (L(\Hsix,1),N)\]
where $L$ is $L(\Hfif,1)$ or $L(\Hsix,1)$.}
\end{prop:LinL=DB'}
\begin{rmk}
    \begin{enumerate}
        \item The subgroups $L(\Hfif,1)$ and $L(\Hsix,1)$ contains $\GL$ modulo $\sim$.
        \item Assume that $\Lin(L)$ is $\Hfif$ or $\Hsix$.
        In the remaining cases,  $L$ is equivalent to $GL_2(\mathbb{R})$ or to $\GR$ with respect to the relation $\sim$ .
        The former reduces to Kobayashi~\cite[Prop.\ A.2.1]{Kob90}, while the latter can never occur in a proper pair.
        Therefore, to complete the classification in this case, it suffices to prove Theorem~\ref{prop:LinL=DB'}.
    \end{enumerate}
\end{rmk}
\begin{proof}[Proof of Thereom~\ref{prop:LinL=DB'}]
    We only prove $L(\Hfif,1)\pf N$, $L(\Hsix,1) \not\pf N$ and $L(\Hsix,1) \not\pf L(\Hsec1,2)$, since other pairs do not satisfy the (CI) condition.

    In order to prove  $L(\Hfif,1) \pf N$ and $L(\Hsix,1) \not\pf N$, we use Theorem~\ref{thm:slded}. Since $N \subset \SR$ and 
    \[L(\Hfif,1)\cap (\SR) = \{\begin{pmatrix} e^t & 0 & u \\ 0 & e^{-t} & 0\end{pmatrix}\colon t,u \in \R\},\]
    \[ L(\Hsix,1) \cap (\SR) = \{ \begin{pmatrix}1 & 0 & 0 \\ u & 1 & v \end{pmatrix} \colon u,v \in \R \}, \]
    we have $L(\Hfif,1) \pf N$ and $L(\Hsix,1) \not\pf N$ by Theorems \ref{thm:properSR} and \ref{thm:slded}.

    We prove $L(\Hsix,1) \not\pf L(\Hsec1,1)$. Since $\Lin(L(\Hsix,1)) \not\pf \Lin(N) $ in $\GR$ by Theorem~\ref{thm:properGL}, there exists a compact set $S$ of $\GL$ and  sequences $t_n,$ $t_n^\prime,$ $u_n^\prime, $ $a_n,$ $b_n,$ $c_n,$ $d_n,$ $A_n,$ $B_n,$ $C_n,$ $D_n$ $\in \R$
    such that
    \begin{itemize}
        \item $t_n \to \infty$ as $n \to \infty$,
        \item $\begin{pmatrix} a_n & b_n \\ c_n & d_n\end{pmatrix}, \begin{pmatrix} A_n & B_n \\ C_n & D_n\end{pmatrix} \in S$,
        \item $\begin{pmatrix}a_n &b_n \\ c_n & d_n \end{pmatrix} \begin{pmatrix}e^{t_n^\prime} &0 \\u_n^\prime & e^{t^\prime_n} \end{pmatrix} \begin{pmatrix}A_n & B_n \\ C_n & D_n \end{pmatrix} =\begin{pmatrix} e^{t_n} & 0 \\ 0 & 1 \end{pmatrix}$.
    \end{itemize}
     When we set $e_n=f_n=F_n=0, E_n=\frac{b_n t_n}{e^{t_n^\prime}}, v_n^\prime= d_n t_n -\frac{b_n t_n}{e^{t_n^\prime}}u_n^\prime$,
    we have 
    \[ \begin{pmatrix} a_n & b_n & e_n \\ c_n &d_n &f_n\end{pmatrix}^{-1} \begin{pmatrix} e^{t_n^\prime} & 0 & 0\\u_n^\prime & e^{t_n^\prime} & v_n^\prime \end{pmatrix} \begin{pmatrix} A_n & B_n & E_n \\ C_n & D_n& F_n \end{pmatrix} = \begin{pmatrix} e^t_n & 0 & 0 \\ 0 & 1 & t_n\end{pmatrix},
    \]
    and the elements $\begin{pmatrix} a_n & b_n &e_n \\ c_n&d_n&f_n\end{pmatrix}, \begin{pmatrix} A_n & B_n & E_n \\ C_n & D_n&F_n\end{pmatrix}$ contained in some compact set of $\GR$ for all $n$. Therefore, we have $L(\Hsix,1) \not\pf N$ in $\GR$.
\end{proof}

\subsection{Proof of Theorem~\ref{thm:slded}}
This subsection is devoted to the proof of Theorem~\ref{thm:slded}.

\newtheorem*{thm:slded}{\rm\bf Theorem~\ref{thm:slded}}
\begin{thm:slded}{\itshape
    Suppose that $L \subset SL_n(\R) \ltimes\R^n$ is a  subset and $L^{\prime} \not\subset SL_n(\R)\ltimes \R^n$ is a connected subgroup of $GL_n(\R)\ltimes \R^n$. Then, we have the following equivalence:
\[ L \pitchfork L^{\prime} \ \text{in  }  GL_n(\R) \ltimes \R^n \iff L\pitchfork (L^{\prime}\cap SL_n(\R) \ltimes \R^n) \text{ in } SL_n(\R) \ltimes \R^n.\]}
\end{thm:slded}

\begin{proof}
The forward direction is obvious, so we prove the opposite direction. Assume $L\pf H$ in $GL_n(\R) \ltimes \R^n$. Then, there exists a compact set $S$ in $GL_n(\R) \ltimes \R^n$, a sequence $g_i^\prime \in H$ diverging to infinity, and a sequence $s_i , s_i^\prime \in S$ such that 
\[ s_i g^\prime_i s^\prime_i \in L. \]
Since $H \not\subset SL_n(\R)\ltimes \R^n$, there exists $X\in \mathrm{Lie}(H)$ such that $\det (\Lin(\exp X))\neq 1$. Define the subset $C$ of $GL_n(\R) \ltimes \R^n$ by
\[ C\coloneqq \{ \exp tX \colon -\log(\max_{s\in S}\det (\Lin(s))) \leq t \leq -\log (\min_{s\in S} \det(\Lin(s)) \},\]
which is a compact set. We have 
\[ s_i g^\prime_i s^\prime_i  = \left(s_i \exp(-t_iX)\right) \left( \exp (t_iX) g_i^\prime \exp(t_i^\prime X) \right) \left(\exp(-t_i^\prime X)s_i^\prime \right),\]
where
\[ \det(\Lin(s_i)) = \det(\Lin(\exp (t_i X))),\quad
 \det (\Lin(s_i^\prime))=\det (\Lin(\exp(t_i^\prime X))).\]
Since the sequences $s_i \exp (-t_iX)$ and $\exp(-t_i^\prime X)s_i^\prime$ are contained in the compact set $CSC \cap (\SR)$, and $\exp (t_i X) g_i^\prime \exp (t_i^\prime X)$ is an unbounded sequence in $H \cap (SL_n(\R)  \ltimes \R^n)$, we have 
\[ L ^\prime \not\pf (H \cap SL_n(\R) \ltimes \R^n ) \text{ in } SL_n(\R) \ltimes  \R^n.\]
\end{proof}

\subsection{Proof of Theorem~\ref{thm:RCI}}
We prove Theorem~\ref{thm:RCI}, which can be adapted for affine transformation groups of any dimension.

\newtheorem*{thm:RCI}{\rm\bf Theorem~\ref{thm:RCI}}
\begin{thm:RCI}{\itshape
Let $L$ be a closed subgroup in $GL_n(\R)\ltimes \R^n$ whose linear part $\Lin(L)$ is closed, and $H$ be a subset in $GL_n(\R)\ltimes \R^n$. If the pair $(L,\R^n)$ is (CI) and $\Lin(L) \pf \Lin(H)$ in $GL_n(\R)$, then $L \pf H$ in $GL_n(\R)\ltimes \R^n$.}
\end{thm:RCI}

For the proof of the theorem, we prepare the following proposition.
\begin{prop}\label{prop:RCIproper}
    In the setting of the above theorem, the pair $(L,\R^n)$ is (CI) if and only if $L \pf \R^n$ in $GL_n(\R)\ltimes \R^n$.
\end{prop}
\begin{proof}
    The if part is obvious, so we prove the converse direction. Since the pair $(L,\R^n)$ is (CI), the subgroup 
    \[ \ker (\Lin|_L) = L\cap \R^n\]
    is compact. Since it is a subgroup of $\R$, it is a trivial group. Thus, we obtain the continuous bijective homomorphism $\Lin|_L\colon L \to \Lin(L)$. By the open mapping theorem (see \cite[Appx.\ 1]{HM06}), the map is a homeomorphism.

    For any compact set $S \subset (GL_n(\R)\ltimes \R^n)/\R^n \cong GL_n(\R)$, we obtain 
    \[ L_S = (\Lin|_L)^{-1}(\Lin|_L)(L_S) = (\Lin|_L)^{-1}(\Lin(L)_S).\]
    The subset $\Lin(L)_S$ is compact. Indeed, the action $\Lin(L) \curvearrowright GL_n(\R)$ is proper since $\Lin(L)$ is a closed subgroup of $GL_n(\R)$. Hence, the subset $L_S$ is also compact. Therefore, the pair $(L,\R^n)$ is proper since the action $L\curvearrowright GL_n(\R)\ltimes \R^n/\R^n$ is proper.
\end{proof}

We now prove Theorem~\ref{thm:RCI}.
\begin{proof}[Proof of Theorem~\ref{thm:RCI}.]
    For a compact subset $S$ in $GL_n(\R)\ltimes\R^n$, the subset 
    \[ C\coloneqq \Lin(L) \cap \Lin(S)\Lin(H)\Lin(S^{-1}) \]
    is compact since $\Lin(L) \pf \Lin(H)$ in $GL_n(\R)$. Hence, we have 
    \begin{align*}
         L\cap SH S^{-1} &\subset L \cap \left( (\Lin(L)\cap \Lin(S)\Lin(H)\Lin(S^{-1}))\times \R^n\right) \\
         &= L \cap (C\times\R^n). 
    \end{align*}
    Since the subset on the right-hand side is compact by Proposition~\ref{prop:RCIproper}, the subset $L\cap SH S^{-1}$ is also compact. Therefore, we obtain $L\pf H$.
\end{proof}


\section{Proof of Theorem~\ref{thm:conditionABC}}
In this section, we prove Theorem~\ref{thm:conditionABC}. We begin with the following remark which helps reduce discussion of determining properness.

\begin{rmk}
    The inclusion relations among $L(\Hsec{\alpha},i)$ and $L(\Hfou{\beta},j)$ are depicted in Figure~\ref{fig:hasse}. The order in the Hasse diagram means that the inclusion can be realized for some choices of the parameters $\alpha$ and $\beta$, up to the equivalence relation $\sim$.
\end{rmk}
\begin{figure}[htbp]
  \centering
  \begin{tikzpicture}[scale=0.9]

  \node (LH47) at (4,6) {$L(B(\beta),7)$};
  \node (LH46) at (-3,4.5) {$L(B(-3),6)$};
  \node (LH25) at (0,4.5)  {$L(A(\alpha),5)$};
  \node (LH44) at (4,4.5)  {$L(B(\beta),4)$};
  \node (LH45) at (7,4.5)  {$L(B(-1),5)$};

  \node (LH43) at (-3,3)   {$L(B(-3),3)$};
  \node (LH23) at (0,3)    {$L(A(\alpha),3)$};
  \node (LH41) at (4,3)    {$L(B(\beta),1)$};
  \node (LH24) at (7,3) {$L(A(1),4)$};
  \node (LH42) at (10,3) {$L(B(1),2)$};

  \node (LH21) at (0,1.5)  {$L(A(\alpha),1)$};
  \node (LH22) at (7,1.5) {$L(A(1),2)$};

  \draw[gray] (LH47) -- (LH46);
  \draw[gray] (LH47) -- (LH25);
  \draw[gray] (LH47) -- (LH44);
  \draw[gray] (LH47) -- (LH45);

  \draw[gray] (LH46) -- (LH43);
  \draw[gray] (LH43) -- (LH21);

  \draw[gray] (LH25) -- (LH23);
  \draw[gray] (LH23) -- (LH21);

  \draw[gray] (LH44) -- (LH41);
  \draw[gray] (LH41) -- (LH21);

  \draw[gray] (LH44) -- (LH23);
  \draw[gray] (LH44) -- (LH42);

  \draw[gray] (LH45) -- (LH24);

  \draw[gray] (LH24) -- (LH22);
  \draw[gray] (LH42) -- (LH22);

  \draw[gray] (LH22) -- (LH23);
  \draw[gray] (LH24) -- (LH25);

  \draw[line width=3pt, white] (LH44) -- (LH42);
  \draw[gray] (LH44) -- (LH42);

  \draw[line width=3pt, white] (LH25) -- (LH24);
  \draw[gray] (LH25) -- (LH24);

  \draw[line width=3pt, white] (LH44) -- (LH23);
  \draw[gray] (LH44) -- (LH23);

  \draw[line width=3pt, white] (LH41) -- (LH21);
  \draw[gray] (LH41) -- (LH21);

\end{tikzpicture}
  \caption{Inclusions among $L(A(\alpha),i),L(B(\beta),j)\subset \GR$}
  \label{fig:hasse} 
\end{figure}

\subsection{{\bf(A)} The case $\Lin(L)=\Hsec\alpha, \Lin(H)=\Hsec\beta$}

In this subsection, we prove Theorem~\ref{thm:conditionABC}~{\bf(A)}. We now review its statement.

\newtheorem*{thm:conditionABC}{\rm\bf Theorem~\ref{thm:conditionABC}}
\begin{thm:conditionABC}{\bf(A)} {\itshape 
Assume that $\Lin(L)=\Hsec\alpha, \Lin(H)=\Hsec{\beta}$. Then, the (CI) condition is equivalent to properness of the pair.}
\end{thm:conditionABC}

We prove this using the classification of the connected subgroups of $\GR$ whose linear part is $\Hsec{\gamma}$.  For the proof of the forward direction, we list the pairs that satisfy the condition (CI).

\begin{lem}\label{lem:CIpair}
The connected pairs satisfying the condition (CI) are as follows.
\begin{enumerate}
    \item \[
\begin{aligned}
&(L(\Hsec{\alpha},1),L(\Hsec{\beta},1)),\quad
(L(\Hsec{\alpha},1),L(\Hsec{\beta},3)),\quad
(L(\Hsec{\alpha},1),L(\Hsec{\beta},5)),\\
&(L(\Hsec{1},2),L(\Hsec{\beta},3)),\quad
(L(\Hsec{1},2),L(\Hsec{\beta},5)),\quad
(L(\Hsec{\alpha},1),L(\Hsec{1},2)),\\
&(L(\Hsec{\alpha},1),L(\Hsec{1},4)),
\end{aligned}
\]
where the parameters have different absolute values.
    \item $(L(\Hsec{1},2), L(\Hsec{1},3))$, $(L(\Hsec1,1), L(\Hsec{1},2))$, $(L(\Hsec1,1),L(\Hsec{1},4))$.
\end{enumerate}

\end{lem}

\begin{proof}[Proof of Theorem~\ref{thm:conditionABC}~{\bf(A)}]
It suffices to prove that the pairs in the above lemma are proper. By Theorem~\ref{thm:RCI}, the pairs in Lemma~\ref{lem:CIpair}~(1) are proper. Thus, we may prove the following pairs are proper:
\begin{enumerate}
    \item $(L(\Hsec{1},1), L(\Hsec{1},4))$,
    \item $(L(\Hsec{1},2), L(\Hsec{1},3)).$
\end{enumerate}
The remaining pair $(L(\Hsec1,1),L(\Hsec1,2))$ is proper since $L(\Hsec1,2) \subset L(\Hsec1,4)$. Assume that the pair (1) is not proper. Then there exists a divergent sequence $\begin{pmatrix} e^{2t_n} & 0 & u_n\\ 0 & 1 & 2t_n\end{pmatrix} \in L(\Hsec1,4)$ and sequences $s_n, s_n^{\prime}$ contained in some compact set of $\GR$ such that 
\[ s_n \begin{pmatrix} e^{2t_n} & 0 & u_n\\ 0 & 1 & 2t_n\end{pmatrix} s_n^{\prime} \in L(\Hsec{1},1)\]
for all $n \in \N$. Consider the translational part. Then, we have 
\[ \begin{pmatrix} u_n\\2t_n\end{pmatrix} + \begin{pmatrix} e^{2t_n} &0 \\0 & 1 \end{pmatrix} \U (s_n^{\prime}) \asymp_+ \Lin(s_n^{-1})\begin{pmatrix} 0 \\ 0\end{pmatrix} \asymp_+ \begin{pmatrix} 0 \\ 0\end{pmatrix}.\]
The second entry of the left hand side is not bounded, but that of the right hand side is bounded, which is contradiction.

Assume the pair (2) is not proper. For proof, we prepare a lemma.
\begin{lem}\label{lem:H2H2lem}
Assume sequences $\begin{pmatrix} e^{2t_n} & 0 \\ 0 & 1\end{pmatrix}, \begin{pmatrix} e^{2t_n^{\prime}} & 0 \\ 0 & 1\end{pmatrix}$ and $\begin{pmatrix} a_n & b_n \\ c_n & d_n\end{pmatrix} ,s_n^{\prime} \in \GL$ contained in some compact set of $\GL$ and satisfy $\begin{pmatrix} a_n & b_n \\ c_n & d_n\end{pmatrix} \begin{pmatrix} e^{2t_n} & 0 \\ 0 & 1\end{pmatrix}(s_n^{¥prime})^{-1} = \begin{pmatrix} e^{2t_n^{\prime}} & 0 \\ 0 & 1\end{pmatrix}.$ Then, we have 
\begin{enumerate}
    \item $t_n\asymp_+t_n^{\prime}$,
    \item $|c_n|\lesssim_\times e^{-2t_n}$ and $|a_n|, |d_n| \asymp_\times 1$.
\end{enumerate}
\end{lem}
\begin{proof}
    The first statement follows by Fact~\ref{fact:ue}. We prove the second one. We have 
    \[ \begin{pmatrix} a_ne^{2t_n} & b_n \\ c_ne^{2t_n} & d_n  \end{pmatrix} = \begin{pmatrix} e^{2t_n^{\prime}} & 0 \\ 0 & 1\end{pmatrix} s_n^{¥prime}.\]
    By Proposition~\ref{prop:normcpt}, the norm of the second column $||\pmat{0&1} s_n^{¥prime} ||$ is asymptotically equal to $1$, that is, $||\pmat{0&1}s_n^{¥prime}|| \asymp_\times 1 $. Thus, the norm of the second column of the left hand side $\Vmat{c_ne^{2t_n}&d_n} \asymp_\times 1$. Hence, we have $|c_n| \lesssim_\times e^{-2t_n}.$ Since $\begin{vmatrix} a_n & b_n \\ c_n & d_n\end{vmatrix} \asymp_\times 1$, we have 
    \[ |a_n| \gtrsim_\times \frac{\begin{vmatrix} a_n & b_n \\ c_n & d_n\end{vmatrix} -|b_n||c_n| }{|d_n|} \gtrsim_\times 1,\quad |d_n| \gtrsim_\times \frac{\begin{vmatrix} a_n & b_n \\ c_n & d_n\end{vmatrix} -|b_n||c_n|}{|a_n|} \gtrsim_\times 1.\]
\end{proof}
We prove Lemma~\ref{lem:CIpair}~(2). Assume the pair is not proper. Then, there exist divergent sequences $g_n = \begin{pmatrix} e^{2t_n} & 0 & 0 \\ 0 & 1 & 2t_n\end{pmatrix} \in L(\Hsec1,2)$, $g_n^{\prime} = \begin{pmatrix} e^{2t_n^{\prime}} & 0 & v_n^{\prime} \\ 0 & 1 & 0\end{pmatrix} \in L(\Hsec1, 3)$ and sequences $s_n =\begin{pmatrix} a_n & b_n & * \\ c_n & d_n & *\end{pmatrix}, s_n^{\prime} $ contained in some compact set of $\GR$ such that 
\begin{enumerate}
    \item[(a)] $\Lin(s_n) \Lin(g_n) \Lin(s_n^{\prime})^{-1} = \Lin(g_n^{\prime}),$
    \item[(b)] $\U(g_n^{\prime}) + \Lin(g_n^{\prime})\U(s_n^{\prime}) \asymp_+ \Lin(s_n)\U(g_n)$.
\end{enumerate}
By the condition (a) and Lemma~\ref{lem:H2H2lem}, we have $t_n\asymp_+t_n^{\prime}$ and $|a_n|, |d_n| \asymp_\times 1$. By (b), we have 
\[\begin{pmatrix} v_n^{\prime} \\ 0\end{pmatrix}+ \begin{pmatrix} e^{2t_n^{\prime}} & 0 \\ 0 & 1 \end{pmatrix} \Lin(s_n^{\prime}) \asymp_+ \frac{1}{|\det s_n|} \begin{pmatrix} a_n & b_n \\ c_n & d_n \end{pmatrix} \begin{pmatrix} 0 \\ 2t_n\end{pmatrix}.\]
The second entry of the left hand side is bounded, but that of right hand side $\frac{1}{\det s_n}\cdot 2d_nt_n \asymp_\times t_n $ is not bounded, which is a contradiction.
\end{proof}

\subsection{{\bf(B)} The case $\Lin(L)=\Hfou\alpha, \Lin(H)=\Hfou\beta$}
In this section, we prove Theorem~\ref{thm:conditionABC}~{\bf(B)}. We review its statement.

\begin{thm:conditionABC}{\bf(B)} {\itshape
Assume that $\Lin(H)=\Hfou\alpha, \Lin(L)=\Hfou\beta$. If the pair $(L,H)$ is proper, then either $L$ or $H$ is $L(\Hfou{-3},3)$ or $L(\Hfou{-3},6)$. Furthermore, if $H$ is one of them, then the following equivalence holds:
        \[ L\pf H \iff (L,H) \text{ is (CI) and } |\alpha|<3.\]
        }
\end{thm:conditionABC}

The former statement follows easily from the (CI) condition. In order to  prove the backward implication of the latter part of the theorem, we prepare the following proposition.

\begin{prop}\label{prop:mt7prop}
Assume that $|\alpha| < |\beta|$, and that the subset
\[
L_{|\det| \asymp_\times 1} := \{ g \in L \mid |\det \Lin(g)| \asymp_\times 1 \}
\]
satisfies $L_{|\det| \asymp_\times 1} \pf L'$. If $L \not\pf L'$, then there exist 
\begin{itemize}
    \item a sequence $g_n \in L$ of the form
    \[
    g_n = \begin{pmatrix} e^{(\alpha+1)t_n} & 0 & * \\ u_n & e^{(\alpha -1)t_n} & * \end{pmatrix},
    \]
    \item a sequence $g^\prime_n \in H$ of the form
    \[
    g_n^\prime = \begin{pmatrix} e^{(\beta+1)t_n^\prime} & 0 & * \\ u_n^\prime & e^{(\beta -1)t_n^\prime} & * \end{pmatrix},
    \]
    \item a sequence $k_n = \begin{pmatrix} a_n & b_n \\ c_n & d_n \end{pmatrix} \in K \subset \GL$, where $K$ is a compact subset of $\GL$,
    \item a sequence $w_n \in \mathbb{R}^2$ with $\|w_n\| \lesssim_\times 1$,
\end{itemize}
such that the following conditions hold:
\begin{enumerate}
\item $t_n \to \infty$
\item  $t_n^\prime \asymp_+ \frac{\alpha}{\beta} t_n $,
\item  $\U(g_n^\prime) + \Lin(g_n^\prime) w_n \asymp_+ k_n \U(g_n)$,
\item  $|u_n^\prime| \gtrsim_\times e^{(\alpha+1)t_n}$,
\item $|b_n| \asymp_\times \frac{e^{(\alpha+1)t_n}}{|u_n^\prime|}$.
\end{enumerate}
\end{prop}

\begin{proof}
 Since $L \not\pf L^{\prime}$, there exist diverging sequences $g_n \in L, g^\prime_n \in L^{\prime}$ such that $g_n \in Sg_n^\prime S^{-1}$ for some compact set $S \subset \GR$. Since $L|_{|\det|\asymp_\times 1} \pf L^{\prime}$, $t_n$ is not bounded. If $t_n\to -\infty $, then we can replace the sequence $g_n$ with its inverse $g_n^{-1}$, and obtain a sequence with $t_n\to \infty$.
 
We look at the asymptotic behavior of the parameters $t_n, s_n$.
By uniform estimate of the Cartan projection (Fact~\ref{fact:ue}), we have 
\begin{align*}
    \frac{1}{2} \log \left( \frac{1}{2}( A_n + \sqrt{A_n^2 -4e^{4\alpha t_n}} )\right) &\asymp_+ \frac{1}{2} \log \left( \frac{1}{2}( B_n + \sqrt{B_n^2 -4e^{4\beta t_n^\prime}} )\right), \\
    \frac{1}{2} \log \left( \frac{1}{2}( A_n - \sqrt{A_n^2 -4e^{4\alpha t_n}} )\right) &\asymp_+ \frac{1}{2} \log \left( \frac{1}{2}( B_n - \sqrt{B_n^2 -4e^{4\beta t_n^\prime}} )\right),
\end{align*}
where the following quantities $A_n, B_n$ depending on $t_n,t_n^\prime,u_n,u_n^{\prime}$ are defined as follows:
\[ A_n \coloneqq e^{2(\alpha +1)t_n} + e^{2(\alpha -1)t_n} + u_n^2, \quad B_n\coloneqq e^{2(\beta +1)t_n^\prime} + e^{2(\beta -1)t_n^\prime} + (u_n^\prime)^2 .\]
Adding the two sides yields the estimate
\begin{align*} 
&\frac12 \log \left( \frac14 (A_n+ \sqrt{A_n^2 -4e^{4\alpha t_n}})(A_n - \sqrt{A_n^2-4e^{4\alpha t_n}} )\right) \\
& \quad \asymp_+ \frac12 \log \left( \frac14 (B_n +\sqrt{B_n^2-4e^{4\beta t_n^\prime}}) (B_n -\sqrt{B_n^2-4e^{4\beta t^\prime_n}})\right).
\end{align*}
Consequently, we obtain $\alpha t_n \asymp_+ \beta t_n^\prime$, which is the estimate (2).
Moreover, as $t_n \to \infty$, 
\begin{align*}
    A_n + \sqrt{A_n^2 -4e^{4\alpha t_n}}
    &\gtrsim_\times A_n + \sqrt{(A_n -2e^{2\alpha t_n})^2} \\
    &= 2(A_n -e^{2\alpha t_n}) \asymp_\times 2e^{2(\alpha +1)t_n},\\
    B_n+\sqrt{B_n^2 -4e^{4\beta t_n^\prime}} &\lesssim_\times 2B_n \asymp_\times 2(e^{2(\beta+1)t_n^\prime }+e^{2(\beta -1)t_n^\prime} +(u_n^\prime)^2).
\end{align*}
Since 
\begin{align*}
    e^{2(\alpha+1)t_n} \gnsim_\times e^{2(\beta+1)\frac{\beta}{\alpha} t_n} \asymp_\times e^{2(\beta +1)t_n^\prime},\\
    e^{2(\alpha+1)t_n} \gnsim_\times e^{2(\beta-1)\frac{\beta}{\alpha} t_n} \asymp_\times e^{2(\beta -1)t_n^\prime}
\end{align*}
for $|\alpha| < |\beta|$, we obtain $e^{(\alpha +1)t_n} \lesssim_\times |u_n^\prime|$, which is the estimate (4).

By comparing the linear parts, we have 
\[ \begin{pmatrix} a_n & b_n \\ c_n& d_n\end{pmatrix}\begin{pmatrix} e^{(\beta +1)t_n^{\prime}} &0\\u_n^\prime & e^{(\beta -1)t_n^\prime} \end{pmatrix} C_n = \begin{pmatrix} e^{(\alpha +1)t_n} & 0 \\ u_n & e^{(\alpha -1)t_n}\end{pmatrix},\]
where $\begin{pmatrix} a_n & b_n \\ c_n & d_n\end{pmatrix}$ and $C_n$ is a sequence contained in a compact set $K$ of $\GL$.
Applying Proposition~\ref{prop:normcpt} to the first column, we obtain
\[
\Vmat{ a_ne^{(\beta+1)t'_n} + b_n u_n^\prime& b_n e^{(\beta -1)t_n^\prime} } \asymp_\times e^{(\alpha +1)t_n} .
\]
Hence,
\[ \Vmat{ a_n\frac{e^{(\beta+1)t'_n}}{e^{(\alpha+1)t_n}} + b_n\frac{u_n^\prime}{e^{(\alpha +1)t_n}} &  b_n\frac{e^{(\beta -1)t_n^\prime}}{e^{(\alpha+1)t_n}} } \asymp_\times 1 . \]
Since
\begin{equation*}
\left| a_n\frac{e^{(\beta +1)t_n^\prime}}{e^{(\alpha +1)t_n}} \right|, \quad \left| b_n\frac{e^{(\beta -1)t_n^\prime}}{e^{(\alpha +1)t_n}} \right| \lnsim_\times 1
\end{equation*}
when $|\alpha|<|\beta|$,
it follows that $|b_n \frac{u_n^\prime}{e^{(\alpha+1)t_n}}|\asymp_\times 1 .$ Hence, we have
$|b_n| \asymp_\times \frac{e^{(\alpha+1)t_n}}{|u_n^\prime|}$.

\end{proof}

Here, we prove Theorem~\ref{thm:conditionABC}~{\bf(B)}.
\begin{proof}[Proof of Theorem~\ref{thm:conditionABC}~{\bf(B)}]
First, we prove the forward direction. Trivially, the proper pairs $(L, H)$ is (CI), so we prove $|\gamma|<3.$ Assume that $|\gamma|\geq 3.$ and $H = L(\Hfou{\gamma},j)$. 

If $j=2,5$, then the parameter $\gamma$ does not satisfy $|\gamma| \geq 3$.

If $j=3,6$, then the pair $(L,H)$ is not (CI).

If $j= 1,4,7$, then the connected subgroups $L(\Hfou{\gamma},j)$ contains the subset
\[ L_0 \coloneqq \{ \begin{pmatrix} e^{-2t} & 0 & 0 \\ 0 & e^{-4t} & 0\end{pmatrix}\colon t\in \R\}\]
modulo $\sim$ for $|\gamma| \geq 3$. Indeed, it follows from Fact~\ref{fact:pc} that 
\[ \Hsec{-3} \subset \Hfou{\gamma} \quad \text{modulo} \sim \text{ in }\GL\]
for $|\gamma| \geq 3$. Since $L(\Hfou{\gamma},j)$ contains $\Hfou{\gamma}\times\{0\}$ when $j=1,4,7$, we obtain the above inclusion. Therefore, it is necessary that $|\gamma|<3$ for the pair $(L,H)$ to be  proper.

We prove the opposite direction. We may prove the following claim. The other  proper pairs can be obtained by the inclusions $L(\Hfou{\gamma},1) \subset L(\Hfou{\gamma},4)$ and $L(\Hfou{-3},3) \subset L(\Hfou{-3},6)$.
\end{proof}
\begin{claim}\label{cl:mt7}
    \begin{enumerate}
        \item $L(\Hfou{-3}, 3) \pf L(\Hfou{\gamma},4)$ if $|\gamma|<3$.
        \item $L(\Hfou{-3},3) \pf L(\Hfou{-1},5).$
        \item $L(\Hfou{\gamma},1)\pf L(\Hfou{-3},6)$ if $|\gamma|<3$.
        \item $L(\Hfou{1},2) \pf L(\Hfou{-3},6)$.
    \end{enumerate}
\end{claim}
\begin{proof}
 (1) We write $L=L(\Hfou{\gamma},4), H= L(\Hfou{-3},3)$. We take sequences 
 \[ g_n= \begin{pmatrix} e^{(\gamma+1)t_n} & 0 & 0 \\ u_n & e^{(\gamma -1) t_n} & v_n\end{pmatrix}\in L, g_n^\prime = \begin{pmatrix} e^{-2t_n^\prime}& 0 & u_n^\prime \\ e^{-2t_n^\prime}u_n^\prime & e^{-4t_n^\prime} & \frac12 (u_n^\prime)^2\end{pmatrix}\in H, \]
 $k_n, w_n$ that satisfy the conditions of Proposition~\ref{prop:mt7prop}. Then, Proposition~\ref{prop:mt7prop}~(2),(4) provides that the parameters $u_n^\prime, t_n^\prime$ satisfy 
 \[ t_n^\prime \asymp_+ -\frac{\gamma}{3}t_n, \quad |e^{-2t_n^\prime} u_n^\prime| \gtrsim_\times 1.\]
 It follows from the condition~(3) in Proposition~\ref{prop:mt7prop} that 
 \[ \begin{pmatrix} 0 \\ v_n\end{pmatrix} +\begin{pmatrix} e^{(\gamma+1)t_n} & 0 \\ u_n & e^{(\gamma -1)t_n}\end{pmatrix} \begin{pmatrix} (w_1)_n\\ (w_2)_n\end{pmatrix} \asymp_+ \begin{pmatrix} a_n & b_n \\ c_n & d_n\end{pmatrix}\begin{pmatrix} u_n^\prime \\ \frac{1}{2} (u_n^\prime)^2 \end{pmatrix}.\tag{*}\]
 The first entry of the right hand side of (*) satisfies
 \[ | a_n u_n^\prime + \frac{1}{2}b_n(u_n^\prime)^2| \gtrsim_\times e^{(\frac{2\gamma}{3}+2)t_n}.\]
 Hence, we obtain by Proposition~\ref{prop:seqasymp}~(4) and the equation (*)
 \[ |(w_1)_n| \gtrsim_\times \frac{e^{(\frac{2\gamma}{3}+2)t_n}}{e^{(\gamma +1)t_n}} =e^{(-\frac{1}{3}\gamma+1)t_n}.\]
 The right hand side diverge to infinity when $|\gamma|<3$, which is a contradiction.

 (2) We write $L=L(\Hfou{-1},5), H= L(\Hfou{-3},3)$. We take sequences 
 \[ g_n= \begin{pmatrix} 1 & 0 &2t_n \\ u_n & e^{-2t_n} & v_n\end{pmatrix}\in L, g_n^\prime=\begin{pmatrix} e^{-2t_n^\prime} & 0 & u_n^\prime \\ e^{-2t_n^\prime} u_n^\prime & e^{-4t_n^\prime} & \frac12 (u_n^\prime)^2\end{pmatrix} \in H, \]
 $k_n, w_n$ such that satisfy the conditions of Proposition~\ref{prop:mt7prop}. Then, the parameters $t_n,u_n^\prime, t_n^\prime$ satisfy 
 \[ t_n^\prime \asymp_+ \frac{t_n}{3}, \quad |e^{-2t_n^\prime}u_n^\prime| \gtrsim_\times 1.\]
 In the same fashion as (1), we obtain
  \[ \begin{pmatrix} 2t_n \\ v_n\end{pmatrix} +\begin{pmatrix} 1 & 0 \\ u_n & e^{-2t_n}\end{pmatrix} \begin{pmatrix} (w_1)_n\\ (w_2)_n\end{pmatrix} \asymp_+ \begin{pmatrix} a_n & b_n \\ c_n & d_n\end{pmatrix}\begin{pmatrix} u_n^\prime \\ \frac{1}{2} (u_n^\prime)^2 \end{pmatrix}.\]
 The first entry of the right-hand side satisfies
 \[ | a_n u_n^\prime + \frac{1}{2}b_n(u_n^\prime)^2| \gtrsim_\times e^{\frac{4}{3}t_n}.\]
 However, that of the left-hand side grows polynomially, which is a contradiction.

 (3) We write $L=L(\Hfou{\gamma},1), H= L(\Hfou{-3},6)$. We take sequences 
 \[g_n= \begin{pmatrix} e^{(\gamma +1)t_n} & 0 & 0 \\ u_n & e^{(\gamma -1)t_n} & 0\end{pmatrix}\in L , g_n^\prime= \begin{pmatrix} e^{-2t_n^\prime} & 0 & u_n^\prime \\ e^{-2t_n^\prime}u_n^\prime & e^{-4t_n^\prime} & v_n^\prime\end{pmatrix} \in H,\]
 $k_n, w_n$ that satisfy the conditions of Proposition~\ref{prop:mt7prop}. Then, the parameters $t_n,u_n^\prime, t_n^\prime$ satisfy 
 \[ t_n^\prime \asymp_+ -\frac{\gamma}{3}t_n, \quad |e^{-2t_n^\prime}u_n^\prime| \gtrsim_\times e^{(\gamma+1)t_n}.\]
Considering the translational parts, we have 
\[  \begin{pmatrix} u_n^\prime \\ v_n^\prime \end{pmatrix} +\begin{pmatrix} e^{-2t_n^\prime} & 0 \\ e^{-2t_n^\prime}u_n^\prime & e^{-4t_n^\prime}\end{pmatrix} \begin{pmatrix} (w_1)_n\\ (w_2)_n\end{pmatrix} \asymp_+ k_n^{-1}\begin{pmatrix} 0 \\ 0\end{pmatrix}.\]
The first entry of the left hand side is not bounded since 
\[ |u_n^\prime| \gtrsim_\times e^{(\frac{\gamma}{3}+1)t_n} \gnsim_\times e^{\frac{2\gamma}{3}t_n} \gtrsim_\times |e^{-2t_n^\prime}(w_1)_n|,\]
and  $|u^\prime_n| \to \infty$ when $|\gamma| <3$. However, that of the right-hand side is bounded, which is a contradiction.

(4) We write $L=L(\Hfou{1},2), H= L(\Hfou{-3},6)$. We take sequences 
\[ g_n = \begin{pmatrix} e^{2t_n} & 0 & 0 \\ u_n & 1 & 2t_n\end{pmatrix}\in L, 
g_n^\prime = \begin{pmatrix} e^{-2t_n^\prime} & 0 & u_n^\prime \\e^{-2t_n^\prime} u_n^\prime &e^{-4 t_n^\prime} & v_n^\prime\end{pmatrix} \in H, \]
$k_n, w_n$ that satisfy the conditions of Proposition~\ref{prop:mt7prop}. Then, the parameters $t_n,u_n^\prime, t_n^\prime$ satisfy 
 \[ t_n^\prime \asymp_+ -\frac{t_n}{3}, \quad |e^{-2t_n^\prime}u_n^\prime| \gtrsim_\times e^{2 t_n}.\]
 Considering the translational part, we have
  \[ \begin{pmatrix} u_n^\prime \\ v_n^\prime \end{pmatrix} +\begin{pmatrix} e^{-2t_n^\prime} & 0 \\ e^{-2t_n^\prime}u_n^\prime & e^{-4t_n^\prime}\end{pmatrix} \begin{pmatrix} (w_1)_n\\ (w_2)_n\end{pmatrix} \asymp_+ k_n^{-1}\begin{pmatrix} 0 \\ 2t_n \end{pmatrix}.\]
  By an analogous argument to (3), the first entry of the left hand side grows exponentially, but that of the right hand side grows at most polynomially, which is a contradiction.
\end{proof}

\subsection{{\bf(C)} The case $\Lin(L)=\Hsec\alpha, \Lin(H)=\Hfou\beta$}
In this section, we prove Theorem~\ref{thm:conditionABC}{\bf(C)}. Considering the inclusion modulo $\sim$, we may prove the following statements.
\begin{prop}\label{prop:mt8}
    \begin{enumerate}[label=(\alph*)]
        \item $L(\Hsec{1},4) \pf L(\Hfou{\beta},1)$ for all $\beta$.
        \item $L(\Hsec{\alpha},1) \pf L(\Hfou{-1},5)$ if and only if $\alpha \neq 0$.
        \item \begin{enumerate}[label=(\arabic*)]
            \item $L(\Hsec{\alpha},3) \pf L(\Hfou{-3},3)$ if and only if $|\alpha| \neq 3$.
            \item $L(\Hsec{1},4) \pf L(\Hfou{-3},3).$
            \item $L(\Hsec{\alpha},1) \pf L(\Hfou{-3},6)$ if and only if $|\alpha|\neq 3$
            \item $L(\Hsec{1},2)\pf L(\Hfou{-3},6).$
        \end{enumerate}
        \item $L(\Hsec{1},2) \pf L(\Hfou{\beta},4)$ if and only if $-1\leq  \beta <1$.
        \item $L(\Hsec\alpha,1) \pf L(\Hfou1,2)$ if and only if $|\alpha| \ge 1$.
        \item \begin{enumerate}[label=(\arabic*)]
            \item $L(\Hsec{\alpha},5) \pf L(\Hfou{1},\beta)$ if $|\alpha|>|\beta|.$
            \item $L(\Hsec{\alpha},1) \pf L(\Hfou{\beta},7)$ if $|\alpha| > |\beta|.$
            \item $L(\Hsec{\alpha},5) \pf L(\Hfou{1},2)$ if $|\alpha|>1$.
            \item $L(\Hsec{\alpha},5) \pf L(\Hfou{-3},3)$ if $|\alpha|>3$.
            \item $L(\Hsec{1},2) \pf L(\Hfou\beta,7)$ if $|\beta| <1.$
        \end{enumerate}
        \begin{enumerate}[label=(\arabic*)$^\prime$]
            \item $L(\Hsec{\alpha},1) \not\pf L(\Hfou{1},\beta)$ if $|\alpha|\le|\beta|.$
            \item $L(\Hsec{\alpha},3) \not\pf L(\Hfou{1},2)$ if $|\alpha| \leq 1$.
            \item $L(\Hsec{\alpha},5) \not\pf L(\Hfou{-3},3)$ if $|\alpha|\leq 3$.
            \item $L(\Hsec{1},2) \not\pf L(\Hfou\beta,7)$ if $|\beta| \geq 1.$
        \end{enumerate}
    \end{enumerate}
\end{prop}

First, we prove the statements (a) and (c).
\begin{proof}[Proof of (a).]
    We know that $L(\Hsec1,4) \pf \GL$ in \cite[Prop.\ A.2.1]{Kob90}, and $L(\Hfou\beta,1) \subset \GL$ for any $\beta$. Therefore, it follows that $L(\Hsec1,4) \pf L(\Hfou\beta,1)$ for any $\beta$.
\end{proof}
\begin{proof}[Proof of (c).]
    Since $L(\Hsec{1},4) \subset L(\Hfou{-1},5), L(\Hsec{1},2) \subset L(\Hfou1,2)$ modulo $\sim$, the statements (2),(4) follows from Claim~\ref{cl:mt7}~(2),(4).

    We prove (1)$^\prime$,(3). If $|\alpha| = 3$, then the pairs are not (CI), and hence not proper. If $|\alpha|>3$, then the pairs are proper by Theorem~\ref{thm:RCI}.
    Assume $|\alpha| < 3$. Then, we have $L(\Hsec\alpha,3) \subset L(\Hfou{\alpha},4)$ and $L(\Hsec\alpha,1) \subset L(\Hfou\alpha,1)$ modulo $\sim$. By Claim~\ref{cl:mt7}~(1) and (3), we obtain $L(\Hsec{\alpha},3) \pf L(\Hfou{-3},3)$ and $L(\Hsec{\alpha},1) \pf L(\Hfou{-3},6)$ if $|\alpha| < 3$. 
\end{proof}

For the proof of (b), we prepare a lemma. The proof is parallel to Proposition~\ref{prop:mt7prop}, so is omitted.

\begin{lem}\label{lem:mt8b}
    Let $|\alpha| \leq 1$. Assume sequences $l_n= \begin{pmatrix} e^{(\alpha +1)t_n} & 0 \\ 0 & e^{(\alpha -1)t_n}\end{pmatrix}$ and $l^\prime_n = \begin{pmatrix} 1 & 0 \\ u_n^\prime & e^{-2t_n^\prime} \end{pmatrix}
    $
    satisfy $l_n \in S l^\prime_nS^{-1}$ for all $n \in \N$ and some compact set $S$ of $\GL$.
    Then, we obtain $t_n^\prime \asymp_+ -\alpha t_n$.
\end{lem}

\begin{proof}[Proof of (b).] 
If $\alpha = 0 $, then we have 
\begin{align*}
    SL_2(\R) &\sim \{ \begin{pmatrix} e^x & 0 & 0 \\ 0 & e^{-x}&0 \end{pmatrix}\colon x \in \R\} = L(\Hsec0,1), \\ 
    SL_2(\R) &\sim \{ \begin{pmatrix} 1 & 0 & 0 \\ x & 1&0 \end{pmatrix}\colon x \in \R\} \subset L(\Hfou{-1},5).
\end{align*}
Therefore, we obtain $L(\Hsec0,1) \not\pf L(\Hfou{-1},5)$.

Assume $\alpha \neq 0$, and $L(\Hsec\alpha,1) \not\pf L(\Hfou{-1},5)$. Then, we obtain sequences $g_n,g_n^\prime$ of the form 
\[ g_n =\begin{pmatrix} e^{(\alpha +1)t_n} & 0 & 0 \\
0 & e^{(\alpha -1)t_n} & 0\end{pmatrix}, \quad g_n^\prime = \begin{pmatrix} 1 & 0 & 2t_n^\prime \\ u_n^\prime & e^{-2t_n^\prime} & v_n^\prime \end{pmatrix},
\]
and $s_n, s_n^\prime \in S$ where $S$ is a compact subset of $\GR$ such that $s_ng_n = g_n^\prime s_n^\prime$. Look at the linear part, and we have $t_n^\prime \asymp_+ -\alpha t_n$ by Lemma~\ref{lem:mt8b}. Look at the translational part, and we have 
\[ \begin{pmatrix} 2t_n^\prime \\ v_n^\prime \end{pmatrix} + \begin{pmatrix} 1& 0 \\ u_n^\prime &e^{-2t_n^\prime}\end{pmatrix} \U(s_n^\prime) \asymp_+ \Lin(s_n)\begin{pmatrix} 0 \\ 0 \end{pmatrix} \asymp_+ \begin{pmatrix} 0 \\ 0\end{pmatrix}.\]
The right-hand side is bounded, whereas the left-hand side is diverges as $t_n\to \infty$ since $t_n^\prime \asymp_+ -\alpha t_n$ and $\alpha \neq 0$. This is a contradiction.
\end{proof}

For the proof of (d), we prepare the following lemma.
\begin{lem}\label{lem:mt8d}
    Assume that sequences $l_n=\begin{pmatrix}e^{2t_n} & 0 \\ 0 & 1 \end{pmatrix}, l=\begin{pmatrix} 1 & 0 \\ u_n^\prime &  e^{-2t_n^\prime}\end{pmatrix} \in \GL$ and $r_n=\begin{pmatrix} a_n & b_n \\ c_n & d_n\end{pmatrix},r_n^\prime$ satisfy that 
    \begin{itemize}
        \item $ t_n\to \infty$,
        \item $l_nr_n=r_n^\prime l_n^\prime$,
        \item the sequences $r_n,r_n^\prime$ are contained in some compact set in $\GL$.
    \end{itemize}
    Then, we have $|b_n| \asymp_\times 1.$
\end{lem}
\begin{proof}
    Since $r_n$ is contained in some compact set in $\GL$, $b_n$ is bounded above.
    By the uniform estimate of Cartan projection~(Fact~\ref{fact:ue}), we obtain $t_n^\prime \asymp_+ -t_n$. We have
    \[ \begin{pmatrix} a_ne^{2t_n} & b_n \\ c_ne^{2t_n} & d_n\end{pmatrix} = \begin{pmatrix} 1 & 0 \\ u_n^\prime & e^{-2t_n^\prime}\end{pmatrix} r_n^\prime.\]
    Considering the norm of the first column, it follows that
    \[ \Vmat{ a_ne^{2t_n} & b_n} \asymp_\times 1.\]
    Thus, we have $|a_n| \lesssim_\times e^{-2t_n}$, and hence
    \[ |b_n|\gtrsim_\times \frac{\begin{vmatrix} a_n & b_n \\c_n & d_n\end{vmatrix}-|a_n||d_n|}{|c_n|} \gtrsim_\times 1.\]
\end{proof}
We prove Proposition~\ref{prop:mt8}~(d).
\begin{proof}[Proof of (d).]
    If $|\beta|<1$, then the pair $(L(\Hsec1, 2), L(\Hfou\beta,4))$ is proper by Theorem~\ref{thm:RCI}.
    If $\beta=1$, then we can check that the pair is not (CI), and hence it is not proper.

    Assume that $\beta = -1$. We prove that the pair is proper. If not, we can take sequences $g_n=\begin{pmatrix}e^{2t_n} & 0 & 0\\0 &1 & 2t_n  \end{pmatrix} \in L(\Hsec1, 2), g_n^\prime = \begin{pmatrix} 1 & 0 & 0 \\ u_n^\prime & e^{-2t_n^\prime} & v_n^\prime \end{pmatrix} \in L(\Hfou\beta,4)$ and $s_n,s_n^\prime$ contained in some compact set of $\GR$ such that $s_ng_n=g_n^\prime s_n^\prime$. By the above lemma, we obtain $|b_n|\asymp_\times 1$. Consider the translational part of the equation $s_ng_n=g_n^\prime s_n^\prime$, we obtain 
    \[ \begin{pmatrix} 0 \\ v_n^\prime\end{pmatrix}+ \begin{pmatrix} 1 & 0  \\ u_n^\prime & e^{-2t_n^\prime} \end{pmatrix}\U(s_n^\prime) \asymp_+ \begin{pmatrix} a_n & b_n \\ c_n & d_n\end{pmatrix}\begin{pmatrix} 0 \\ 2t_n \end{pmatrix},\]
    where we write $\Lin(s_n) =\begin{pmatrix} a_n & b_n \\ c_n & d_n\end{pmatrix}$.
    The first entry of the left-hand side is bounded, but that of the right-hand side $2b_nt_n$ is not bounded by Lemma~\ref{lem:mt8d}, which is a contradiction. Therefore, we obtain  $L(\Hsec1, 2) \pf L(\Hfou{-1},4)$ in $\GR$.

    Finally, we prove $(L(\Hsec1, 2), L(\Hfou\beta,4))$ is not proper if $|\beta| >1$. Since $\Hsec1 \not\pf  \Hfou\beta$ in  $\GL$ by properness criterion (Fact~\ref{fact:pc}), there exist sequences 
    \[ l_n=\begin{pmatrix} e^{2t_n} & 0 \\ 0 & 1\end{pmatrix}, l_n^\prime= \begin{pmatrix} e^{(\beta +1)t_n^\prime} & 0 \\ u_n^\prime & e^{(\beta -1)t_n^\prime}\end{pmatrix}\]
    and 
    \[ r_n = \begin{pmatrix} a_n & b_n \\ c_n & d_n\end{pmatrix}, r_n^\prime\]
    contained in some compact set in $\GL$, such that $t_n\to \infty$ and $r_nl_n = l_n^\prime r_n^\prime$. By Lemma~\ref{lem:mt8b}, we obtain $t_n^\prime \asymp_+ \frac{t_n}{\beta}$ and $u_n^\prime \asymp_\times\exp(2t_n)$. We set 
    \[ g_n = \begin{pmatrix} e^{2t_n} & 0 & 0 \\ 0 & 1 & 2t_n\end{pmatrix}, g_n^\prime = \begin{pmatrix} e^{(\beta +1) t_n^\prime} & 0 & 0 \\ u_n^\prime & e^{(\beta -1) t_n^\prime} & d_nt_n - u_n^\prime \frac{b_nt_n}{e^{(\beta +1)t_n^\prime}}\end{pmatrix}, \]
    \[ s_n= \left(
\begin{array}{cc}
r_n& \begin{matrix} 0 \\ 0 \end{matrix}
\end{array}
\right), s_n^\prime =\left( \begin{array}{cc} r_n^\prime &\begin{matrix}  \frac{b_nt_n}{e^{(\beta+1)t_n^\prime}}  \\ 0\end{matrix}\end{array}\right).\]
    Then, it follows that $s_ng_n = g_n^\prime s_n^\prime$. Moreover, $\frac{b_nt_n}{e^{(\beta+1)t_n^\prime}}$ is bounded as $t_n\to \infty$. Indeed, we have
    \[ \left|\frac{b_n t_n}{e^{(\beta +1)t_n^\prime}}\right| \asymp_\times \frac{|b_n||t_n|}{e^{( \frac{1}{\beta} +1)t_n}} \lesssim_\times 1.  \]
    The last evaluation holds since $\frac{1}{\beta} +1> 0$ for $|\beta| >1$. Therefore, we have $L(\Hsec1,2) \not\pf L(\Hfou\beta,4)$ for $|\beta| >1$.
\end{proof}

Before the proof of (e), we prepare a lemma considering the asymptotic behavior of the linear parts. This lemma follows immediately by the uniform estimate of Cartan projection (Fact~\ref{fact:ue}), so we omit its proof.

\begin{lem}\label{lem:abs't'}
    Assume that  the sequences $l_n,l_n^\prime$ are of the form
    \[     l_n=\begin{pmatrix} e^{(\alpha+1)t_n} & 0 \\ 0& e^{(\alpha-1)t_n} \end{pmatrix}, \quad
        l_n^\prime=\begin{pmatrix}  e^{(\beta+1)t_n^\prime} & 0 \\ u_n^\prime & e^{(\beta-1)t_n^\prime} \end{pmatrix},\]
    and satisfies $r_nl_n= l_n^\prime r_n^\prime$ for some sequences $r_n,r_n^\prime$ contained in some compact set of $\GL$, then:
   \begin{enumerate}
\item If $\lvert\alpha\rvert<\lvert\beta\rvert$, we have the following estimates
\[
  t_n^\prime \asymp_+ \frac{\alpha}{\beta}\,t_n, 
  \qquad
  u_n^\prime \asymp_\times \exp\!\big((\alpha +1)t_n\big).
\]

\item If $\lvert\alpha\rvert=\lvert\beta\rvert$, we still have
\[
  t_n^\prime \asymp_+ \frac{\alpha}{\beta}\,t_n,
\]
and moreover
\[
  u_n^\prime \lesssim_\times \exp\!\big((\alpha +1)t_n\big).
\]
\end{enumerate}
\end{lem}

\begin{proof}[Proof of (e).]
By Theorem~\ref{thm:RCI}, we obtain $L(\Hsec\alpha,1) \pf L(\Hfou1,2)$ if $|\alpha|>1$. 

We prove $L(\Hsec1,1) \pf L(\Hfou1,2)$. Assume the contrary, we can take the sequences $g_n=\begin{pmatrix} e^{2t_n} & 0& 0\\ 0 & 1 &0 \end{pmatrix} \in L(\Hsec1,1), g_n^\prime=\begin{pmatrix} e^{2t_n^\prime} & 0 & 0\\ u_n & 1 & 2t_n^\prime\end{pmatrix}\in L(\Hfou1,2)$, and sequences $s_n, s_n^\prime$ contained in some compact set in $\GL$, such that $s_n g_n=g^\prime_n s^\prime_n$ for all $n\in \N$. Consider the translational parts, we obtain
\[ \begin{pmatrix} e^{2t^\prime_n} & 0 \\ u_n^\prime & 1\end{pmatrix}\begin{pmatrix} (w_1)_n \\ (w_2)_n\end{pmatrix}+\begin{pmatrix} 0 \\ 2t^\prime_n\end{pmatrix} \asymp_+ \Lin(s_n)\begin{pmatrix}0 \\ 0 \end{pmatrix} \asymp_+ \begin{pmatrix} 0 \\ 0\end{pmatrix},\]
where we write $\begin{pmatrix} (w_1)_n\\(w_2)_n\end{pmatrix} = \U (s_n^\prime)$. Note that $(w_i)_n$ is a bounded sequence.
Since $|(e^{2t_n^\prime})(w_1)_n|$ is bounded, we have $|(u_n^\prime) (w_1)_n| \lesssim_\times |(e^{2t_n^\prime})(w_1)_n| \lesssim_\times 1$ by Lemma~\ref{lem:abs't'}~(2). Thus, the second entry of the left hand side $(u_n^\prime) (w_1)_n+ (w_2)_n+ 2t_n^\prime \asymp_+ 2t_n^\prime$ is not bounded. However, that of the right hand side is bounded, which is a contradiction. Since $L(\Hsec1,1) \sim L(\Hsec{-1},1)$, we also obtain $L(\Hsec{-1},1) \pf L(\Hfou1,2)$.

We prove $L(\Hsec\alpha,1) \not\pf L(\Hfou1,2)$ if $|\alpha|<1$. If $|\alpha| = 0$, they have non-compact intersection . Assume $0<|\alpha|<1$. Then we obtain sequences $l_n,l_n^\prime,r_n,r_n^\prime$ satisfying the assumption in Lemma~\ref{lem:abs't'}. We set 
    \[ g_n= \begin{pmatrix} l_n & \begin{matrix} 0 \\ 0\end{matrix}\end{pmatrix} ,
    \quad g_n^\prime= \begin{pmatrix} l_n^\prime & \begin{matrix} 0 \\ 2t_n^\prime\end{matrix} \end{pmatrix}\]
    \[ s_n=\begin{pmatrix} r_n &\begin{matrix} -2\frac{t_n^\prime e^{2t_n^\prime}}{u_n^\prime} \\ 0 \end{matrix} \end{pmatrix} , \quad s_n^\prime= \begin{pmatrix} r_n^\prime & \begin{matrix} -\frac{2t_n^\prime}{u_n^\prime} \\ 0\end{matrix}\end{pmatrix}.\]
    Then, we have $s_ng_n=g_n^\prime s_n^\prime$ and 
    \begin{align*} \left| \frac{-2t_n^\prime}{u_n^\prime} \right|&\asymp_\times \frac{|\alpha||t_n|}{\exp((\alpha +1)t_n)} \lnsim_\times 1, \\
    \left| -2 \frac{t^\prime_n e^{2t_n^\prime}}{u_n^\prime}\right| &\asymp_\times \alpha t_n\exp((\alpha -1)t_n) \lnsim_\times 1
    \end{align*}
    by Lemma~\ref{lem:abs't'}~(2) for $|\alpha|<1$.
    Therefore, $s_n,s_n^\prime$ are contained in some compact set, and hence we have $L(\Hsec\alpha,1) \not\pf L(\Hfou1,2)$ if $ 0 < |\alpha| <1$.

\end{proof}
At the end of this section, we prove Proposition~\ref{prop:mt8}~(f).
\begin{proof}[Proof of (f).]
    The statements (1),$\cdots$,(5) immediately follow by Theorem~\ref{thm:RCI}.

    The statement (1)$^\prime$ holds since the linear parts are not proper in $\GL$ by the properness criterion (Fact~\ref{fact:pc}) and the translational parts are trivial.

    The statement (2)$^\prime$ holds when $|\alpha|=0,1$ since the two subgroups have a non-compact intersection, and follows by (e) since $L(\Hsec\alpha, 1) \subset L(\Hsec\alpha,3)$ even when $0<|\alpha|<1$.

    We prove (4)$^\prime$. By the properness criterion (Fact~\ref{fact:ue}), there exist sequences
    \[l_n=\begin{pmatrix}  e^{2t_n} & 0 \\ 0 & 1\end{pmatrix} \in \Hsec{1}, \quad l_n^\prime = \begin{pmatrix}  e^{(\beta +1)t_n^\prime} & 0 \\ u_n^\prime & e^{(\beta -1)t_n^\prime} \end{pmatrix}\in \Hfou{\beta},\]
    and $r_n, r_n^\prime$ contained in some compact set in $\GL$ such that 
    \[ t_n \to \infty , \quad r_nl_n= l_n^\prime r_n^\prime.\]
    We set 
    \[ s_n=\begin{pmatrix}  r_n &\begin{matrix}  0 \\ 0\end{matrix}\end{pmatrix},\quad s_n^\prime=\begin{pmatrix} r_n^\prime& \begin{matrix}0 \\ 0 \end{matrix}\end{pmatrix},
    \]
    \[ g_n= \begin{pmatrix} e^{2t_n} & 0 & 0 \\ 0 & 1& 2t_n\end{pmatrix}, \quad g_n^\prime=\begin{pmatrix} \begin{matrix}e^{(\beta +1)t_n^\prime} & 0 \\ u_n^\prime & e^{(\beta -1)t_n^\prime} \end{matrix} & r_n\begin{pmatrix}0 \\ 2t_n \end{pmatrix}\end{pmatrix}
    \]
    Then, we have $s_ng_n=g_n^\prime s_n^\prime,$ $g_n \in L(\Hsec1,2)$ and $g_n^\prime \in L(\Hfou\beta,7)$ . Thus, we have $L(\Hsec1,2) \not\pf L(\Hfou\beta, 7)$. 

    The proof of (3)$^\prime$ is parallel to the above discussion .

   
\end{proof}

\appendix
\section{Tables of properness for connected subgroup pairs}
This section provides a complete summary of the properness for  all pairs of connected subgroups, presented in table. When a condition involving parameters is written in a cell, it indicates a necessary and sufficient condition for the properness of the corresponding pair.

\begin{table}[h]
\caption{Properness of pairs in the case where $\Lin(L)=\Hfir,  \Lin(H)\subset \SL$}
\label{tab:A11}
\centering
\renewcommand{\arraystretch}{1.5}
\begin{tabular}{c@{\quad}cccccc}
\toprule
\diagbox{$L$}{$H$} & $SL_2(\R)$ & $\R^2$ & $S$ & $L$ & $M$ & $N$ \\
\midrule
$L(\Hfir,1)$ & $\pf$ & $\pf$ & $\pf$ & $\pf$ & $\pf$ &  $\pf$\\
$L(\Hfir,2)$ & $\pf$ &  &  & &  & $\pf$ \\
\bottomrule
\end{tabular}
\end{table}

\begin{table}[h]
\caption{Properness of pairs in the case where $\Lin(L)=\Hsec\gamma, \Lin(H) \subset \SL$}
\label{tab:A1.2}
\centering
\renewcommand{\arraystretch}{1.5}
\begin{tabular}{c@{\quad}cccccc}
\toprule
\diagbox{$L$}{$H$} & $SL_2(\R)$ & $\R^2$ & $S$ & $L$ & $M$ & $N$ \\
\midrule
$L(\Hsec0,1)$  &  & $\pf$ &  &  & $\pf$ & $\pf$ \\
$L(\Hsec\gamma,1)$ if $\gamma \neq 0$ &  $\pf$ & $\pf$ & $\pf$ & $\pf$ & $\pf$ & $\pf$ \\
$L(\Hsec{1},2)$ & $\pf$ & $\pf$ & $\pf$ & $\pf$ & $\pf$ & $\pf$ \\
$L(\Hsec0,3)$ &  &  &  &  &  & $\pf$ \\
$L(\Hsec\gamma,3)$ if $\gamma\neq 0$ & $\pf$ &  &  &  &  & $\pf$ \\
$L(\Hsec{1},4)$ &  $\pf$  &  &  &  &  & $\pf$ \\
$L(\Hsec0,5)$ &    &  &  &  &  &  \\
$L(\Hsec{\gamma},5)$ if $\gamma \neq 0$ & $\pf$ &  &  &  &  & $\pf$ \\
\bottomrule
\end{tabular}
\end{table}

\begin{table}[h]
\caption{Properness of pairs in the case where $\Lin(L)=\Hthi, \Lin(H) \subset \SL$}
\label{tab:A1.3}
\centering
\renewcommand{\arraystretch}{1.5}
\begin{tabular}{c@{\quad}cccccc}
\toprule
\diagbox{$L$}{$H$} & $SL_2(\R)$ & $\R^2$ & $S$ & $L$ & $M$ & $N$ \\
\midrule
$L(\Hthi,1)$ & $\pf$ & $\pf$ & $\pf$ & $\pf$ & $\pf$ & $\pf$ \\
$L(\Hthi,2)$ & $\pf$ &  &  &  &  & $\pf$ \\
$L(\Hthi,3)$ & $\pf$ &  &  &  &  & $\pf$ \\
\bottomrule
\end{tabular}
\end{table}

\begin{table}[hbtp]
\caption{Properness of pairs in the case where $\Lin(L)=\Hfou\gamma, \Lin(H) \subset \SL$}
\label{tab:A1.4}
\centering
\renewcommand{\arraystretch}{1.5}
\begin{tabular}{c@{\quad}cccccc}
\toprule
\diagbox{$L$}{$H$} & $SL_2(\R)$ & $\R^2$ & $S$ & $L$ & $M$ & $N$ \\
\midrule
$L(\Hfou{\gamma},1)$ &  & $\pf$ &  &  & $\pf$ & $\pf$ \\
$L(\Hfou{1},2)$ &  & $\pf$ &  &  & $\pf$ & $\pf$ \\
$L(\Hfou{-3},3)$ & $\pf$ & $\pf$ & $\pf$ & $\pf$ & $\pf$ & $\pf$ \\
$L(\Hfou{\gamma},4)$ &  &  &  &  &  &  \\
$L(\Hfou{-1},5)$ &  &  &  &  &  &  \\
$L(\Hfou{-3},6)$ & $\pf$ &  &  &  &  &  \\
$L(\Hfou{\gamma},7)$ &  &  &  &  &  &  \\
\bottomrule
\end{tabular}
\end{table}

\begin{table}[hbtp]
\caption{Properness of pairs in the case where $\Lin(L)=\Hfif, \Lin(H) \subset \SL$}
\label{tab:A1.5}
\centering
\renewcommand{\arraystretch}{1.5}
\begin{tabular}{c@{\quad}cccccc}
\toprule
\diagbox{$L$}{$H$} & $SL_2(\R)$ & $\R^2$ & $S$ & $L$ & $M$ & $N$ \\
\midrule
$L(\Hfif,1)$ &  &  &  &  &  &  \\
\bottomrule
\end{tabular}
\end{table}

\begin{table}[hbtp]
\caption{Properness of pairs in the case where $\Lin(L)=\Hsix, \Lin(H) \subset \SL$}
\label{tab:A1.6}
\centering
\renewcommand{\arraystretch}{1.5}
\begin{tabular}{c@{\quad}cccccc}
\toprule
\diagbox{$L$}{$H$} & $SL_2(\R)$ & $\R^2$ & $S$ & $L$ & $M$ & $N$ \\
\midrule
$L(\Hsix,1)$ &  &  &  &  &  & $\pf$ \\
\bottomrule
\end{tabular}
\end{table}

\begin{table}[hbtp]
\caption{Properness of pairs in the case where $\Lin(L)=\Hfir, \Lin(H) = \Hsec\gamma$}
\label{tab:A1.7}
\centering
\renewcommand{\arraystretch}{1.5}
\begin{tabular}{c@{\quad}ccccc}
\toprule
\diagbox{$L$}{$H$} & $L(\Hsec{\gamma},1)$ & $L(\Hsec1,2)$ & $L(\Hsec{\gamma},3)$ & $L(\Hsec1,4)$ & $L(\Hsec{\gamma},5)$ \\
\midrule
$L(\Hfir,1)$ & $\pf$ & $\pf$ & $\pf$ & $\pf$ & $\pf$ \\
$L(\Hfir,2)$ & $\pf$ & $\pf$ &  &  &  \\
\bottomrule
\end{tabular}
\end{table}

\begin{table}[hbtp]
\caption{Properness of pairs in the case where $\Lin(L)=\Hfir, \Lin(H) = \Hthi$}
\label{tab:A1.8}
\centering
\renewcommand{\arraystretch}{1.5}
\begin{tabular}{c@{\quad}ccc}
\toprule
\diagbox{$L$}{$H$} & $L(\Hthi,1)$ & $L(\Hthi,2)$ & $L(\Hthi,3)$ \\
\midrule
$L(\Hfir,1)$ & $\pf$ & $\pf$ & $\pf$ \\
$L(\Hfir,2)$ & $\pf$ &  &  \\
\bottomrule
\end{tabular}
\end{table}

\begin{table}[hbtp]
\caption{Properness of pairs in the case where $\Lin(L)=\Hfir, \Lin(H) = \Hfou\gamma$}
\label{tab:A1.9}
\centering
\renewcommand{\arraystretch}{1.5}
\begin{adjustbox}{max width=\linewidth} 
\begin{tabular}{c@{\quad}ccccccc}
\toprule
\diagbox{$L$}{$H$} 
& $L(\Hsec{\gamma},1)$ 
& $L(\Hfou{1},2)$ 
& $L(\Hfou{-3},3)$ 
& $L(\Hfou{\gamma},4)$ 
& $L(\Hfou{-1},5)$ 
& $L(\Hfou{-3},6)$ 
& $L(\Hfou{\gamma},7)$ \\
\midrule
$L(\Hfir,1)$ & $\pf$ & $\pf$ & $\pf$ & $\pf$ & $\pf$ &  $\pf$& $\pf$ \\
$L(\Hfir,2)$ & $\pf$ & $\pf$ & $\pf$ &  &  &  &  \\
\bottomrule
\end{tabular}
\end{adjustbox}
\end{table}

\begin{table}[hbtp]
\caption{Properness of pairs in the case where $\Lin(L)=\Hsec\gamma, \Lin(H) = \Hthi$}
\label{tab:A1.10}
\centering
\renewcommand{\arraystretch}{1.5}
\begin{tabular}{c@{\quad}ccc}
\toprule
\diagbox{$L$}{$H$} & $L(\Hthi,1)$ & $L(\Hthi,2)$ & $L(\Hthi,3)$ \\
\midrule
$L(\Hsec{\gamma},1)$ & $\pf$ & $\pf$ & $\pf$ \\
$L(\Hsec{1},2)$ & $\pf$ & $\pf$ & $\pf$ \\
$L(\Hsec{\gamma},3)$ & $\pf$ &  &  \\
$L(\Hsec{1},4)$ & $\pf$ &  &  \\
$L(\Hsec{\gamma},5)$ & $\pf$ &  &  \\
\bottomrule
\end{tabular}
\end{table}

\begin{table}[hbtp]
\caption{Properness of pairs in the case where $\Lin(L)=\Hthi, \Lin(H) = \Hfou\gamma$}
\label{tab:A1.11}
\centering
\renewcommand{\arraystretch}{1.5}
\begin{adjustbox}{max width=\linewidth} 
\begin{tabular}{c@{\quad}ccccccc}
\toprule
\diagbox{$L$}{$H$} 
& $L(\Hfou{\gamma},1)$ 
& $L(\Hfou{1},2)$ 
& $L(\Hfou{-3},3)$ 
& $L(\Hfou{\gamma},4)$ 
& $L(\Hfou{-1},5)$ 
& $L(\Hfou{-3},6)$ 
& $L(\Hfou{\gamma},7)$ \\
\midrule
$L(\Hthi,1)$ & $\pf$ & $\pf$ & $\pf$ & $\pf$ & $\pf$ & $\pf$ & $\pf$ \\
$L(\Hthi,2)$ & $\pf$ & $\pf$ & $\pf$ &  &  &  &  \\
$L(\Hthi,3)$ & $\pf$ & $\pf$ & $\pf$ &  &  &  &  \\
\bottomrule
\end{tabular}
\end{adjustbox}
\end{table}

\begin{table}[hbtp]
\caption{Properness of pairs in the case where $\Lin(L)=\Hsec\alpha, \Lin(H) = \Hsec\beta$}
\label{tab:A1.12}
\centering
\renewcommand{\arraystretch}{1.5}
\begin{tabular}{c@{\quad}ccccc}
\toprule
\diagbox{$L$}{$H$} 
& $L(\Hsec{\beta},1)$ 
& $L(\Hsec{1},2)$ 
& $L(\Hsec{\beta},3)$ 
& $L(\Hsec{1},4)$ 
& $L(\Hsec{\beta},5)$ \\
\midrule
$L(\Hsec{\alpha},1)$ & $|\alpha|\ne|\beta|$ & $\pf$ & $|\alpha|\ne|\beta|$ & $\pf$ & $|\alpha|\ne|\beta|$ \\
$L(\Hsec{1},2)$ & $\pf$ &  & $|\beta| \ne 1$ &  & $|\beta|\ne 1$ \\
$L(\Hsec{\alpha},3)$ & $|\alpha|\ne|\beta|$ & $|\alpha| \ne 1$ &  &  &  \\
$L(\Hsec{1},4)$ & $\pf$ &  &  &  &  \\
$L(\Hsec{\alpha},5)$ & $|\alpha|\ne|\beta|$ & $|\alpha|\ne 1$ &  &  &  \\
\bottomrule
\end{tabular}
\end{table}

\begin{table}[hbtp]
\caption{Properness of pairs in the case where $\Lin(L)=\Hsec\alpha, \Lin(H) = \Hfou\beta$}
\label{tab:A1.13}
\centering
\renewcommand{\arraystretch}{1.5}
\begin{adjustbox}{max width=\linewidth} 
\begin{tabular}{c@{\quad}ccccccc}
\toprule
\diagbox{$L$}{$H$} 
& $L(\Hfou{\beta},1)$ 
& $L(\Hfou{1},2)$ 
& $L(\Hfou{-3},3)$ 
& $L(\Hfou{\beta},4)$ 
& $L(\Hfou{-1},5)$ 
& $L(\Hfou{-3},6)$ 
& $L(\Hfou{\beta},7)$ \\
\midrule
$L(\Hsec{\alpha},1)$ 
& $|\alpha|>|\beta|$ 
& $|\alpha|\ge 1$ 
& $|\alpha|\ne 3$ 
& $|\alpha|>|\beta|$ 
& $\alpha\ne 0$ 
& $|\alpha|\ne 3$ 
& $|\alpha|>|\beta|$ \\
$L(\Hsec{1},2)$ 
& $\pf$ 
&  
& $\pf$ 
& $-1 \le \beta < 1$ 
&  
& $\pf$ 
& $1 > |\beta|$ \\
$L(\Hsec{\alpha},3)$ 
& $|\alpha|>|\beta|$ 
& $|\alpha|>1$ 
& $|\alpha|\ne 3$ 
&  
&  
&  
&  \\
$L(\Hsec{1},4)$ 
& $\pf$ 
&  
&  $\pf$
&  
&  
&  
&  \\
$L(\Hsec{\alpha},5)$ 
& $|\alpha|>|\beta|$ 
& $|\alpha|>1$ 
& $|\alpha|>3$ 
&  
&  
&  
&  \\
\bottomrule
\end{tabular}
\end{adjustbox}
\end{table}

\begin{table}[hbtp]
\caption{Properness of pairs in the case where $\Lin(L)=\Hfou\alpha, \Lin(H) = \Hfou\beta$}
\label{tab:A1.14}
\centering
\renewcommand{\arraystretch}{1.5}
\begin{adjustbox}{max width=\linewidth} 
\begin{tabular}{c@{\quad}ccccccc}
\toprule
\diagbox{$L$}{$H$} 
& $L(\Hfou{\beta},1)$ 
& $L(\Hfou{1},2)$ 
& $L(\Hfou{-3},3)$ 
& $L(\Hfou{\beta},4)$ 
& $L(\Hfou{-1},5)$ 
& $L(\Hfou{-3},6)$ 
& $L(\Hfou{\beta},7)$ \\
\midrule
$L(\Hfou{\alpha},1)$ &  &  & $|\alpha| < 3$ &  &  & $|\alpha| < 3$ &  \\
$L(\Hfou{1},2)$ &  &  & $\pf$ &  &  & $\pf$ &  \\
$L(\Hfou{-3},3)$ & $|\beta| < 3$ & $\pf$ & & $|\beta| < 3$ & $\pf$ &  &  \\
$L(\Hfou{\alpha},4)$ &  &  & $|\alpha| < 3$ &  &  &  &  \\
$L(\Hfou{-1},5)$ &  &  & $\pf$ &  &  &  &  \\
$L(\Hfou{-3},6)$ & $|\beta| < 3$ & $\pf$ &  &  &  &  &  \\
$L(\Hfou{\alpha},7)$ &  &  &  &  &  &  &  \\
\bottomrule
\end{tabular}
\end{adjustbox}
\end{table}

\clearpage
\bibliography{references} 
\bibliographystyle{myamsalpha}

\begin{flushleft}
\textsc{Shunsuke Miyauchi}\\
Graduate School of Mathematical Sciences, The University of Tokyo\\
3-8-1 Komaba, Meguro-ku, Tokyo 153-8914, Japan\\
Email: \texttt{miyashun@ms.u-tokyo.ac.jp}
\end{flushleft}

\end{document}